\documentclass{elsarticle}

\usepackage{amsmath}
\usepackage{amsthm}
\usepackage{amsfonts}
\usepackage{graphicx}
\usepackage{epstopdf}

\graphicspath{{Figures/}}

\newtheorem{theorem}{Theorem}[section]

\newtheorem{lemma}[theorem]{Lemma}

\newtheorem{corollary}[theorem]{Corollary}
\newtheorem{assumption}[theorem]{Assumption}
\newtheorem{remark}[theorem]{Remark}

\providecommand{\norm}[1]{\left|  \left|  #1 \right| \right|}

\newcommand{\secref}[1]{\S \ref{#1}}

\newcommand{\figref}[1]{Figure \ref{#1}}

\newcommand{\DkF}[1]{(F_{n+\frac{#1}{S}}-F_n)}
\begin{document}

\begin{frontmatter}
\title{New efficient substepping methods for exponential timestepping.}
\author[HW]{G. J. Lord\fnref{fn1}}
\ead{G.J.Lord@hw.ac.uk}
\author[HW]{D. Stone\corref{cor1}\fnref{fn2}}
\ead{ds218@hw.ac.uk}
\cortext[cor1]{Corresponding Author}
\address[HW]{Department of Mathematics \& Maxwell Institute, Heriot-Watt University, Edinburgh}

\fntext[fn1]{Telephone: +44 (0)131 451 8196}
\fntext[fn2]{Telephone: +44 (0)131 451 3432}

\begin{abstract}
Exponential integrators are time stepping schemes which exactly solve
the linear part of a semilinear ODE system. This class of schemes
requires the approximation of a matrix exponential in every step, and
one successful modern method is the Krylov subspace projection
method. We investigate the effect of
breaking down a single timestep into arbitrary multiple substeps,
recycling the Krylov subspace to minimise costs. 
For these recyling based schemes we analyse the local error, 
investigate them numerically and 
show they can be applied to a large system with $10^6$ unknowns.
We also propose a new second order integrator that is found using the extra
information from the substeps to form a corrector to increase the
overall order of the scheme. This scheme is seen to compare favourably
with other order two integrators.
\end{abstract}
\end{frontmatter}

\section{Introduction}

\label{kryrec_intro}
We consider the numerical integration of a large system of semilinear ODEs of the form
\begin{equation}
\frac{du}{dt} =  Lu + F(t,u(t)) \qquad u(0)=u_0, \quad t \in [0, \infty)
\label{semilin}
\end{equation}
with $u,F(t,u(t)) \in \mathbb{R}^N$ and $L\in \mathbb{R}^{N\times N}$
a matrix. Equation \eqref{semilin} arises, for example, from the spatial
discretisation of reaction-diffusion-advection equations. An
increasingly popular method for approximating the 
solution of semlinar ODE systems such as \eqref{semilin} are
exponential integrators. 
These are a class of schemes which approximate \eqref{semilin} by
exactly solving the linear part and are characterised by requiring the
evaluation or approximation of a matrix exponential function of $L$ at
each timestep. A major class of exponential integrators are the
multistep Exponential Time Differencing (ETD) schemes, first developed
in \cite{cm}, other classes include the Exponential Euler Midpoint
method \cite{LEM} and Rosenbrock type methods \cite{rosenbrock_pre,
  rosenbrock}. For an overview of exponential integrators see \cite{Overview,wright-minchev-overview}
and other useful references can be found in  \cite{KT4th}. 

Approximating the matrix exponential and functions of it (like
$\varphi-$functions in \eqref{phigen} below) is a notoriously
difficult problem \cite{19-ways}. A classical technique is Pad\'e
approximation, which 
is only efficient for small matrices. Modern methods range from Taylor
series methods making sophisticated use of scaling and squaring for
efficiency, \cite{Higham-scal}, to approximation with Faber or
Chebyschev polynomials \cite{moret_faber}, \cite{Overview} \S 4.1,
interpolation on Leja points \cite{fast-leja, relpm, mar-par,
  lej-kry-comp, Lejaref}, to Krylov subspace projection techniques
\cite{saad92, hoch-kry, phivpaper, NW, Tok2010kry} which is what we
consider here.

Our schemes are based on the standard exponential integrator ETD1, which can be written as
\begin{equation}
u_{n+1}^{etd} = u_{n}^{etd} + \Delta t \varphi_1(\Delta t L)\left(Lu_n^{etd} +F_n^{etd} \right).
\label{etd1}
\end{equation}
where $ \varphi_1(\Delta t L)$ is defined shortly; $u_n^{etd}
\approx u(t_n)$ at discrete times $t_n= n \Delta t$ for fixed $\Delta
t>0$, $F_n^{etd} \equiv F(t_n,u_n^{etd})$0 and $n\in \mathbb{N}$.
ETD1 is globally first order, and is derived from \eqref{semilin} by variation of constants and approximating $F(t,u(t))$ by the constant $F_n^{etd}$ over one timestep. See for example \cite{Overview, wright-minchev-overview, KT4th, myThesis} for more detail.
It is useful to introduce the additional notation
$$
g(t) \equiv Lu(t) + F(t, u(t)) \quad \text{and} \quad g_n^{etd} \equiv Lu_n^{etd} +F(t_n, u_n^{etd}).
$$
The function $\varphi_1$ is part of a family of matrix exponential functions defined by
$\varphi_0(z) = e^z$, $\varphi_1(z) = z^{-1} \left( e^z - I \right),$ 
and in general
\begin{equation}
\varphi_{k+1}(z) = z^{-1}\left( \varphi_{k+1} - \frac{I}{k!} \right),
\label{phigen}
\end{equation}
where $I$ is the identity matrix. These $\varphi-$functions appear in all exponential integrator schemes; see \cite{NW}. In particular we use $\varphi_1$, and for brevity we introduce the following notation
\begin{equation}
p_{\tau} \equiv \tau \varphi_1(\tau L).
\label{pdef}
\end{equation}
We can then re-write (\ref{etd1}) as
\begin{equation}
u_{n+1}^{etd}= u_{n}^{etd} + p_{\Delta t} g_n^{etd} 
\label{etd1again}
\end{equation}
We consider the Krylov projection method for approximating terms like
$p_{\Delta t} g_n^{etd}$ in (\ref{etd1}). In the Krylov method, this 
term is approximated on a Krylov subspace defined by the vector $g_n$
and the matrix $L$. Typically the subspace is recomputed, in the form
of a matrix of basis vectors $V_m$, every time the solution vector,
$u_n$ in (\ref{etd1}), is updated (and thus also $g_n$). This is done
using a call to the Arnoldi algorithm \cite{saad92}, and is often the
most expensive part of each step. It is possible to `recycle' this
matrix at least once, as demonstrated in \cite{2steppap} for the
exponential Euler method (EEM) (see \cite{Overview} and \eqref{eq:EEM}).
In this paper we investigate this possibility further and use it to construct
new methods based on ETD1 and in \ref{App:EEM} we show how to
construct the general recycling method for EEM.

We examine the effect of splitting the single step of (\ref{etd1}) of
length $\Delta t$ in to $S$ substeps of length $\delta t =
\frac{\Delta t}{S}$, through which the Krylov subspace and matrices
are recycled. By deriving expressions for the local error, we show
that the scheme remains locally second order for any number $S$ of
substeps, and that the leading term of the local error decreases. 
This gives a method based on recyling the Krylov subspace for $S$
substeps. We then obtain a second method
using the extra information from the substeps to form a corrector to
increase the overall order of the scheme. 

The paper is arranged as follows. In Section \ref{sec-kry-proj} we
describe the Krylov 
subspace projection method for approximating the action of
$\varphi-$functions on vectors. In Section \ref{sec:recyc} we describe
the concept 
of recycling the Krylov subspace across substeps in order to increase
the accuracy of the ETD1 based scheme, and show that the leading term of the local error of the scheme decreases as the number of substeps uses increases. We then prove a lemma to express the local error expression at arbitrary order. With this information about the local error expansion, and the extra information from the substeps taken, it is possible to construct correctors for the scheme the increase the accuracy and local order of the scheme. We demonstrate one simple such corrector in Section \ref{cor-sec}. Numerical examples demonstrating the effectiveness of this scheme are presented in Section \ref{sec:numres}.

\section{The Krylov Subspace Projection Method and ETD1}
\label{sec-kry-proj}
We describe the Krylov subspace projection method for approximating $\varphi_1(\Delta t L)$ in (\ref{etd1}). We motivate this by showing how the leading powers of $\Delta t L$ in $L$ are captured by the subspace. 
The series definition of $\varphi_1(\Delta t L)$ is,
\begin{equation}
\varphi_1(\Delta t L) \equiv \sum_{k=0}^{\infty}{\frac{(\Delta t L)^k}{(k+1)!}}.
\label{phi1-2}
\end{equation}
The challenge in applying the scheme (\ref{etd1}) is to efficiently compute, or approximate, the action
of $\varphi_1$ on the vector $g_n^{etd}$. 
The sum  in (\ref{phi1-2}) is useful in motivating a polynomial Krylov subspace approximation. The $m$-dimensional Krylov subspace for the matrix $L$ and vector $g \in \mathbb{R}^N$ is defined by:
\begin{equation}
\mathcal{K}_m(L,g)=\mbox{span}\{g , Lg, \ldots, L^{m-1}g \}.
\label{krydef1}
\end{equation}
Approximating sum in (\ref{phi1-2}) by the first $m$ terms is equivalent to approximation in the subspace $\mathcal{K}_m(L,g_n^{etd})$ in (\ref{krydef1}). We now review some simple results about the general subspace $\mathcal{K}_m(L,g)$, with arbitrary vector $g$, before using the results with $g = g_n^{etd}$ to demonstrate how they are used in the evaluation of (\ref{etd1}). \\
The Arnoldi algorithm (see e.g. \cite{NW}, \cite{saad92}) is used to
produce an orthonormal basis $\{ v_1, \ldots, v_n \}$ for the space
$\mathcal{K}_m(L,g)$ such that 
\begin{equation}
\mbox{span}\{v_1 , v_2, \ldots, v_m \}= \mbox{span}\{g , Lg, \ldots, L^{m-1}g \} .
\label{krydef2}
\end{equation}
It produces two matrices $V_m \in \mathbb{R}^{N\times m}$, whose columns are the $v_k$, and an upper Hessenburg matrix $H_m \in \mathbb{R}^{m \times m}$. The matrices $L$,$H_m$ and $V_m$ are related by
\begin{equation}
L = V_m^T H_m V^T,
\label{HandAprime}
\end{equation}
see, for example, equation (2) in \cite{saad92}, of which \eqref{HandAprime} is a consequence. From \eqref{HandAprime} it follows that,
\begin{equation}
V_mH_mV_m^T =V_mV_m^T LV_mV_m^T.
\label{HandA}
\end{equation}
For any $x \in \mathcal{K}_m(L,g)$, 
$$
V_mV_m^Tx = x,
$$
since $V_mV_m^Tx$ represents the orthogonal projection into the space $\mathcal{K}_m(L,g)$. Therefore, since $L^kg \in \mathcal{K}_m(L,g_n)$, we also have that
$$
V_mV_m^T L^k g = L^k g \qquad \text{for } \qquad 0\leq k \leq m-1
$$
We now consider the relationship between $L^k g$ and $V_mH_m^kV_m^T g$.
\begin{lemma}
Assume $0 \leq k \leq m-1$. Then for $H_m$, $V_m$ corresponding to the Krylov subspace $\mathcal{K}(L,g)$,
\begin{equation}
  V_mH_m^kV_m^T g  = L^k g.
\end{equation}
\label{Lkrybasic}
\end{lemma}
\begin{proof}
We have that $V_mH_m^kV_m^T = (V_mH_mV_m^T)^k$ since $V_m^TV_m=I$. Let $\pi \equiv V_mV_m^T$, the projector into $\mathcal{K}(L,g)$, so that \eqref{HandA} can be more briefly written $V_mH_mV_m^T = \pi L \pi $. Then by (\ref{HandA}) we find
\begin{equation}
V_mH_m^kV_m^T g =  (\pi L \pi )^kg =  (\pi L \pi)^{k-1}\pi L \pi g =
(\pi L \pi)^{k-1} Lg =  \ldots  = L^k g. 
\label{handl}
\end{equation}
\end{proof}
Now consider using the vector $g = g_n^{etd}$, to generate the subspace $\mathcal{K}_m(L,g_n^{etd})$, and the corresponding matrices $H_m$, $V_m$, by the Arnoldi algorithm. By Lemma \ref{Lkrybasic} we have that, up to $k=m$,
$$
V_mH_m^kV_m^T g_n^{etd}  = L^k g_n^{etd}.
$$
Thus, inserting the approximation $L \approx V_mH_m^kV_m^T$ in $\varphi_1(\Delta t L)$ the first $m$ terms are correctly approximated. The Krylov approximation is then
\begin{equation}
\begin{split}
\Delta t \varphi_1(\Delta t L)g_n \approx \Delta t \varphi_1(\Delta t V_mH_mV_m^T)g_n &=  \Delta t V_m\varphi_1(\Delta t H_m)V_m^Tg_n \\
 &= ||g_n||\Delta t V_m\varphi_1(\Delta t H_m)e_1.
\end{split}
\label{phi1kry}
\end{equation}
Let us introduce a shorthand notation for the Krylov approximation of
the $\varphi-$function. Analogous to \eqref{pdef}, for $\tau \in
\mathbb{R}$ let
\begin{equation}
\tilde{p}_{\tau} \equiv \tau V_m \varphi_1(\tau H_m) V_m^T \approx p_{\tau} .
\label{ptildedef}
\end{equation}
Using \eqref{phi1kry} and \eqref{ptildedef} we then approximate
(\ref{etd1again}) by $u_{n+1}^{etd}= u_{n}^{etd} + \tilde{p}_{\Delta t}g_n^{etd}.$
The key here is that the $\varphi_1(\Delta t H_m)$ now needs to be evaluated instead of $\varphi_1(\Delta t L)$. $m$ is chosen such that $m  \ll N$, and a classical method such as a rational Pad\'e is used for $\varphi_1(\Delta t H_m)$, which would be prohibitively expensive for $\varphi_1(\Delta t L)$ for large $N$. \\
One step of the ETD1 scheme (\ref{etd1}), under the approximation $\varphi_1(\Delta t L) \approx V_m\varphi_1(\Delta t H_m)V_m^t$, becomes
\begin{equation}
\begin{split}
u_{n+1} &=  u_{n} + \Delta t V_m\varphi_1(\Delta t H)V_m^Tg_n \\
         &= ||g_n|| \Delta t V_m\varphi_1(\Delta t H)e_1,
\end{split}
\label{etd1kry}
\end{equation}
where $e_1$ is the first basis vector in $\mathbb{R}^m$.
\section{Recycling the Krylov subspace}
\label{sec:recyc}
In the Krylov subspace projection method described in
\secref{sec-kry-proj}, the subspace $\mathcal{K}_m(L,g_n)$ and thus
the matrices $H_m$ and $V_m$ depend on $g_n$. At each step it is
understood that a new subspace must be formed, and $H_m$, $V_m$ be
re-generated by the Arnoldi method, since $g_n$ changes. In
\cite{2steppap} it is demonstrated that splitting the timestep into
two substeps, and recycling $H_m$ and $V_m$, i.e. recycling the Krylov
subspace, can be viable (in that it does not decrease the local order
of the scheme, and apparently decreases the error). We expand on this
concept with a more detailed analysis of the effect of this kind of
recycled substepping applied to the locally second order ETD1 scheme
(\ref{etd1}) (EEM is considered in \ref{App:EEM}). 
We replace a single step of length $\Delta t $ of (\ref{etd1kry}) with
$S$ substeps of length $\delta t$, such that $\Delta t = S \delta
t$. We denote the approximations used in this scheme analogously to
the notation for ETD1 earlier, without the $etd$ superscript, and
introduce the following notation to keep track of substeps
$$
u_{n+\frac{j}{S}} \approx u(t_n + j \delta t)
\qquad \text{and} \qquad 
F_{n+\frac{j-1}{S}} \approx F(t_n + j \delta t, u_{n+\frac{j}{S}} ).
$$
For $j=1$ we calculate $H_m$, $V_m$, from $g_n$,
\begin{equation}
u_{n+\frac{1}{S}} =  u_{n} + \delta t V_m\varphi_1(\delta t H_m)V_m^Tg_n,
\label{recstart}
\end{equation}
and for the remaining $S-1$ steps,
\begin{equation}
u_{n+\frac{j}{S}} =  u_{n+\frac{j-1}{S}} + \delta t V_m\varphi_1(\delta t H_m)V_m^T\left(  Lu_{n+\frac{j-1}{S}} +F_{n+\frac{j-1}{S}}\right) \mbox{,     }1<j\leq S,
\label{recnext}
\end{equation}
where the matrices $H_m$ and $V_m$ \emph{are not re-calculated for any substep}, $j >1$. We call substeps of the form (\ref{recnext}) `recycled steps' and substeps of the form (\ref{recstart}) `initial steps'.  The approximation to $u(t_n  + \Delta t)$ at the end of the step of length $\Delta t $ is then given by
\begin{equation}
u_{n+1} = u_{n+\frac{S-1}{S}} +\delta t V_m\varphi_1(\delta t H_m)V_m^T\left(  Lu_{n+\frac{S-1}{S}} +F_{n+\frac{S-1}{S}}\right).
\label{recS}
\end{equation}
The recycling steps (\ref{recstart}), (\ref{recnext}) can be succinctly expressed using the definition of $\tilde{p}_{\tau}$;
\begin{equation}
u_{n+\frac{1}{S}} =  u_{n} +  \tilde{p}_{\delta t} g_n,
\label{recstartsuccinct}
\end{equation}
\begin{equation}
u_{n+\frac{j}{S}} =  u_{n+\frac{j-1}{S}} +  \tilde{p}_{\delta t}  \left(  Lu_{n+\frac{j-1}{S}} +F_{n+\frac{j-1}{S}}\right) \mbox{,     }1<j\leq S.
\label{recnextsuccinct}
\end{equation}

\subsection{The local error of the recycling scheme}
We now derive an expression for the local error of the scheme defined
by (\ref{recstart}), (\ref{recnext}) and prove that the leading term
decreases with the number of substeps $S$. We use the local error
assumption that $u_n = u(t_n)$ and make an assumption about the
accuracy of the initial Krylov approximation with respect to the error
of the scheme. Let the error in the polynomial Krylov approximation
over a single step (including the error from the approximation of
$\varphi_1(\tau H_m)$ using, e.g. Pad\'e approximation), with subspace
dimension of size $m$, be given by $\mathcal{E}_{n+1}^m$, so that, 
\begin{equation}
\tilde{p}_{\tau}g_n = p_{\tau}g_n + \mathcal{E}_{n+1}^m.
\label{kerr_eq}
\end{equation}
\begin{assumption}
The Krylov approximation error $\mathcal{E}_{n+1}^m$ is much less than the error of  ETD1, and thus does not affect the leading term of the local error of ETD1.
\label{asskryerr}
\end{assumption}
Bounds on $\mathcal{E}_{n+1}^m$ can be found in for example \cite{hoch-kry, saad92}. Practically, we can always reduce $\Delta t$ or increase $m$ until Assumption \ref{asskryerr} is satisfied. \\
For the local error of the recycling scheme, the following result will be used.
\begin{lemma}
\label{L3}
For any $\tau_1,\tau_2 \in \mathbb{R}$, and any vector $v \in \mathbb{R}^N$, 
$$
p_{\tau_1}v + p_{\tau_2}\left( Lp_{\tau_1}v + v \right) = p_{\tau_1 + \tau_2}v,
$$
and the same relation holds for the Krylov approximations, that is,
$$
\tilde{p}_{\tau_1}v + \tilde{p}_{\tau_2}\left( L\tilde{p}_{\tau_1}v + v \right) = \tilde{p}_{\tau_1 + \tau_2}v.
$$
\end{lemma}
\begin{proof}
We prove the second equation. The first can be proved using an almost identical argument, replacing $\tilde{p}_{\tau}$ by $p_{\tau}$ where appropriate. \\
By the definitions of $\tilde{p}_{\tau}$, and $\varphi_1$, i.e. $\tilde{p}_{\tau} = V\varphi_1(\tau H_m)V_m^T = V_mH_m^{-1}\left(e^{\tau_2 H_m}-I\right)V_m^T$ we have
$$
\tilde{p}_{\tau_2}\left( L\tilde{p}_{\tau_1}v + v \right) = V_mH_m^{-1}\left(e^{\tau_2 H_m}-I\right)V_m^T\left(  LV_mH_m^{-1}\left(e^{\tau_1 H_m}-I\right)V_m^T +I \right)v.
$$
By \eqref{HandAprime} this becomes 
$
\tilde{p}_{\tau_2}\left( L\tilde{p}_{\tau_1}v + v \right) =V_mH_m^{-1}\left(e^{ (\tau_2+\tau_1) H_m}-e^{\tau_1 H_m}\right)V_m^T v.
$
Now using the definition of $\tilde{p}_{\tau_1}$,
\begin{equation*}
\begin{split}
&\tilde{p}_{\tau_1}v + \tilde{p}_{\tau_2}\left( L\tilde{p}_{\tau_1}v + v \right) \\
&= V_mH_m^{-1}\left(e^{\tau_1 H_m}-I\right)V_m^Tv + V_mH_m^{-1}\left(e^{ (\tau_2+\tau_1) H_m}-e^{\tau_1 H_m}\right)V_m^T v \\
&=  V_mH_m^{-1}\left(e^{(\tau_2+\tau_1) H_m}-I\right)V_m^Tv,
\end{split}
\end{equation*}
which is $\tilde{p}_{\tau_1 + \tau_2}v$ as desired.
\end{proof}
Without recycling substeps, a single ETD1 step (\ref{etd1}) of length $\Delta t$, using the polynomial Krylov approximation, would be:
\begin{equation}
u_{n+1}^{etd}= u_{n}^{etd} + \tilde{p}_{\Delta t}g_n^{etd}.
\label{etd1brv}
\end{equation}
To examine the local error we compare $u_{n+1}^{etd}$ with the $u_{n+1}$ obtained after some number $S$ of recycled substeps. We can write
$$
u_{n+1}= u_{n} + \tilde{p}_{\Delta t}g_n + R_{n+1}^S,
$$
where $R_{n+1}^S$ represents the deviation from (\ref{etd1brv}) over one step. 
Then we have:
\begin{lemma}
The approximation $u_{n+\frac{j}{S}}$ produced by $j$ substeps of the recycling scheme (\ref{recstartsuccinct}), (\ref{recnextsuccinct}), satisfies
\begin{equation}
u_{n+\frac{j}{S}}=u_n +  \tilde{p}_{j\delta t}g_n + R_{n+\frac{j}{S}}^S ,
\label{herer}
\end{equation}
with 
\begin{equation}
R_{n+\frac{j}{S}}^S  = \sum_{k=1}^{j}{(I+\tilde{p}_{\delta t}L)^{j-k}\tilde{p}_{\delta t}\DkF{k-1}}.
\label{Rdef}
\end{equation}
\label{lemmaForR}
\end{lemma}
\begin{proof}
By induction. For $j=1$, $u_{n+\frac{1}{S}}$ is given by (\ref{recstartsuccinct}) and $R_{n+\frac{1}{S}}^S =0$. Equation (\ref{Rdef}) gives $R_{n+\frac{1}{S}}^S = \tilde{p}_{\delta t}\DkF{0} = 0 $ as required.\\
Assume now (\ref{herer}) holds for some $j\geq1$. Then $u_{n+\frac{j+1}{S}}$ is obtained by a step of (\ref{recnextsuccinct}).
Using (\ref{herer}) we find,
$$
u_{n+\frac{j+1}{S}}= u_{n+\frac{j}{S}} + \tilde{p}_{\delta t}  \left(  Lu_n + L\tilde{p}_{j\delta t}g_n + LR_{n+\frac{j}{S}}^S   +F_{n+\frac{j}{S}}\right),
$$
and since $Lu_n = g_n - F(t_n,u(t_n))$, by the induction hypothesis, 
\begin{equation}
\begin{split}
u_{n+\frac{j+1}{S}} &= u_{n+\frac{j}{S}} + \tilde{p}_{\delta t}  \left(  g_n + L\tilde{p}_{\delta t}g_n + LR_{n+\frac{j}{S}}^S   +F_{n+\frac{j}{S}} - F(t_n,u(t_n))\right) \\ \nonumber
                                &= u_{n+\frac{j}{S}} + \tilde{p}_{\delta t}\left( g_n + L\tilde{p}_{\delta t}g_n \right) + \tilde{p}_{\delta t}LR_{n+\frac{j}{S}}^S + \tilde{p}_{\delta t}\DkF{j}  \\ \nonumber
                                &= u_n +  \tilde{p}_{j\delta t}g_n+\tilde{p}_{\delta t}\left( g_n + L\tilde{p}_{\delta t}g_n \right) + (I+\tilde{p}_{\delta t}L)R_{n+\frac{j}{S}}^S + \tilde{p}_{\delta t}\DkF{j}.  \nonumber
\end{split}
\end{equation}
Thus by Lemma \ref{L3} we have that,
$$
u_{n+\frac{j+1}{S}}= u_n +  \tilde{p}_{(j+1)\delta t}g_n + \tilde{p}_{\delta t}\DkF{j} + (I+\tilde{p}_{\delta t}L)R_{n+\frac{j}{S}}^S .
$$
To complete the proof we need to show:
\begin{equation}
R_{n+\frac{j+1}{S}}^S = \tilde{p}_{\delta t}\DkF{j} + (I+\tilde{p}_{\delta t}L)R_{n+\frac{j}{S}}^S. 
\label{R_ind}
\end{equation}
By the induction hypothesis that (\ref{Rdef}) holds for $j$,
\begin{equation}
\begin{split}
 &\tilde{p}_{\delta t}\DkF{j} + (I+\tilde{p}_{\delta t}L)R_{n+\frac{j}{S}}^S \\
                    &= \tilde{p}_{\delta t}\DkF{j}+ (I+\tilde{p}_{\delta t}L)\sum_{k=1}^{j}{(I+\tilde{p}_{\delta t}L)^{j-k}\tilde{p}_{\delta t}\DkF{j-1}} \\
                    &= \sum_{k=1}^{j+1}{(I+\tilde{p}_{\delta t}L)^{j+1-k}\tilde{p}_{\delta t}\DkF{k-1}} = R_{n+\frac{j+1}{S}}^S.
\end{split}
\end{equation}
Hence the lemma is proved.
\end{proof}
Using (\ref{herer}) we now express the leading order term of the local error in terms of $S$. First we examine the leading order term of $R_{n+1}^S$. 
\begin{lemma}
The term $R_{n+\frac{j}{S}}^S$ in Lemma \ref{lemmaForR}, when expanded in powers of $\Delta t$, satisfies
\begin{equation}
R_{n+\frac{j}{S}}^S = \frac{j(j-1)}{2}\delta t^2 V_mV_m^T  \frac{dF(t_n, u_n)}{dt}  + O(\Delta t^3).
\label{RinLemma}
\end{equation}
\end{lemma}
\begin{proof}
By induction. Since $\DkF{0}=0$, then, $R_{n+\frac{1}{S}}^S = 0$, so that \eqref{RinLemma}  is true for $j=1$. Now assume the result holds for some $j$. Then we can express the term $F_{n+\frac{j}{S}}$ follows:  
\begin{equation}
\begin{split}
F(t_{n+\frac{j}{S}},u_{n+\frac{j}{S}}) &= F(t_{n+\frac{j}{S}},u_n + (j\delta t g_n + O(\delta t^2)))\\ \nonumber                               
&= F(t_{n},u_n ) + j\delta t\frac{\partial F}{\partial t}(t_{n},u_n) + j\delta t\frac{\partial F}{\partial u}(t_{n},u_n)g_n + O(\delta t^2) \\ \nonumber
&= F(t_{n},u_n ) + j\delta t\frac{d F}{d t}(t_{n},u_n) + O(\delta t^2). \nonumber
\end{split}
\end{equation}
We thus have that 
\begin{equation}
\DkF{j} = j\delta t\frac{d F}{d t}(t_{n},u_n) + O(\delta t^2).
\label{DF_leading}
\end{equation}
We then insert \eqref{DF_leading} into the inductive expression (\ref{RinLemma}) for $R_{n+\frac{j}{S}}^S$ and use the expansion $ \tilde{p}_{\delta t} =  \delta t V_mV_m^T + O(\delta t^2)$ to get
$$
R_{n+\frac{j+1}{S}}^S = \delta t V_mV_m^T j\delta t\frac{d F}{d t}(t_{n},u_n) + (I+\delta t V_mV_m^T L)R_{n+\frac{j}{S}}^S  + O(\delta t^3).
$$
Using the induction assumption \eqref{RinLemma},
$$
R_{n+\frac{j+1}{S}}^S = \delta t V_mV_m^T j\delta t\frac{d F}{d t}(t_{n},u_n) + \frac{j(j-1)}{2}\delta t^2 V_mV_m^T  \frac{dF(t_n, u_n)}{dt} + O(\Delta t ^3).
$$
Noting that $\Delta t = S \delta t$ we can write $O(\Delta t^3)$ as $ O(\delta t^3)$. Collecting terms we have, 
$$
R_{n+\frac{j+1}{S}}^S =  \left(\frac{j(j-1)}{2}+j\right)\delta t^2 V_mV_m^T  \frac{dF(t_n, u_n)}{dt} + O(\Delta t ^3). 
$$
The lemma follows since $\frac{j(j-1)}{2}+j = \frac{j(j+1)}{2}$.
\end{proof}
The leading local error term of the ETD1 scheme without substeps is well known to be $\frac{\Delta t^2}{2} \frac{dF(t)}{dt}$ (see \cite{RKEI}), so that we can finally recover the leading term from Lemma \ref{lemmaForR}.
\begin{corollary}
The leading term of the recycling scheme after $j$ steps is
\begin{equation}
u_{n+\frac{j}{S}} =u_{n} + j\delta t g_n + \frac{j^2 \delta t^2}{2}Lg_n + \frac{j(j-1)}{2}\delta t^2 V_mV_m^T  \frac{dF(t_n, u_n)}{dt} + O(\delta t^3).
\label{rectay}
\end{equation}
\end{corollary}
\begin{corollary}
The local error $u(t_n+\Delta t)-u_{n+1}$ of an ETD1 Krylov recycling scheme is second order for any number $S$ of recycled substeps.
Moreover, the local error after $j$ recycled steps is 
$$u(t_n+ j \delta t)-u_{n+\frac{j}{S}}=\frac{(j\delta t)^2}{2} \left( I - \frac{j-1}{j}V_mV_m^T  \right)\frac{df(t)}{dt} + O(\delta t^3).$$
In particular
\begin{equation}
u(t_n+\Delta t)-u_{n+1} =  \frac{\delta t^2}{2} \left(S^2- S(S-1)V_mV_m^T \right)\frac{df(t)}{dt} + O(\delta t^2),
\label{leadrecddt}
\end{equation}
or in terms of $\Delta t$
\begin{equation}
u(t_n+\Delta t)-u_{n+1} =\frac{\Delta t^2}{2} \left(I- \frac{S-1}{S}V_mV_m^T \right)\frac{df(t)}{dt} + O(\Delta t^2).
\label{leadrecdt}
\end{equation}
\label{finalTayErr}
\end{corollary}
It is interesting to compare \eqref{leadrecdt} with the leading term of the local error of regular ETD1, $\frac{\Delta t^2}{2}\frac{df}{dt}$. Since $V_mV_m^T$ is the orthogonal projector into $\mathcal{K}$, then we can see that the $ \frac{\Delta t^2}{2} \frac{S-1}{S} V_mV_m^T \frac{df(t)}{dt}$ part in \eqref{leadrecdt} is the projection of the ETD1 error into $\mathcal{K}$, multiplied by a factor $\frac{S-1}{S} \leq 1$. Thus, in the leading term, according to \eqref{leadrecdt}, the recycling scheme reduces the error of ETD1 by effectively eliminating the part of the error which lives in $\mathcal{K}$. In the limit $S \rightarrow \infty$, the entirety of the error in $\mathcal{K}$ will be eliminated. The effectiveness of the recycling scheme therefore depends on how much of $\frac{df(t)}{dt} $ can be found in $\mathcal{K}$. \\
Corollary \ref{finalTayErr} shows that using $S>1$ recycled substeps is advantageous over the basic ETD1 scheme, in the sense of reducing the magnitude of the leading local error term, whenever
\begin{equation}
\norm{\left(I- \frac{S-1}{S}V_mV_m^T \right)\frac{df(t)}{dt} }< \norm{\frac{df(t)}{dt}},
\label{leadIneq}
\end{equation}
where $\norm{\cdot}$ is a given vector norm. We show in Lemma
\ref{anotherlemma} that increasing $S$ will decrease the Euclidean norm $||\cdot||_2$ of the leading term of the local error. First we require a result on $V_mV_m^T$, the projector into the Krylov Supspace $\mathcal{K}$.

\begin{remark}
Let $x\neq 0$ be a vector such that $V_mV_m^Tx \neq 0$, then for $\alpha \in \mathbb{R}$
\begin{equation}
\norm{\left(I - \alpha V_mV_m^T \right)x}_2^2 = \norm{x}_2^2 + [(1-\alpha)^2 -1]\norm{V_mV_m^Tx}_2^2.
\label{oProjGen}
\end{equation}
\end{remark}

\begin{proof}
An elementary result for orthogonal projectors (see, e.g. \cite{saadbook}) is that
\begin{equation}
\norm{x}_2^2 = \norm{V_mV_m^Tx}_2^2 + \norm{(I-V_mV_m^T)x}_2^2,
\label{oProj}
\end{equation}
which follows from $V_mV_m^Tx \bot (I-V_mV_m^T)x$ (the orthogonality of $V_mV_m^Tx$ and $(I-V_mV_m^T)x$) and the definition of the Euclidean norm. Equation \eqref{oProjGen} is a generalisation of \eqref{oProj} as can be shown as follows. \\
Write $x - \alpha V_mV_m^T x = (I-V_mV_m^T)x + (1-\alpha)V_mV_m^Tx$,
and then, noting that $(I-V_mV_m^T)x \bot (1-\alpha)V_mV_m^Tx$, we see that
$$
\norm{x - \alpha V_mV_m^T x}_2^2 = \norm{(I-V_mV_m^T)x}_2^2 + (1-\alpha)^2\norm{V_mV_m^Tx}_2^2.
$$
Using (\ref{oProj}) to substitute for $\norm{(I-V_mV_m^T)x}_2^2$ yields (\ref{oProjGen}).
\end{proof}

\begin{lemma}
Assume $ \frac{df(t)}{dt} \neq 0$ and $ V_mV_m^T \frac{df(t)}{dt}
\neq 0$. Let $E_{S_1}$ be the local error using the recycling scheme
over a timestep of length $\Delta t$ with $S_1 \geq 1$ substeps, and
$E_{S_2}$ the local error with $S_2$ substeps with $S_2 > S_1$. 
Then,
$$
\norm{E_{S_2}}_2 < \norm{E_{S_1}}_2.
$$
\label{anotherlemma}
\end{lemma}
\begin{proof}
The local errors $E_{S_k}$, $k=1,2$ are given in Corollary
\ref{finalTayErr}. Let $\frac{S_k-1}{S_k} \equiv \beta_k$, $k=1,2$. We
need to show that  
$$
\norm{\left(I-\beta_2 V_mV_m^T \right)\frac{df(t)}{dt}}_2 < \norm{\left(I-\beta_1 V_mV_m^T \right)\frac{df(t)}{dt}}_2.
$$
Let $x \equiv \left(I-\beta_1 V_mV_m^T \right)\frac{df(t)}{dt}$, then $\left(I-\beta_2 V_mV_m^T \right)\frac{df(t)}{dt} = x - \left( \frac{\beta_1 - \beta_2}{\beta_1-1} \right)V_mV_m^Tx$ (showing this involves using $V_mV_m^TV_mV_m^T = V_mV_m^T$). Letting $\gamma \equiv \frac{\beta_1 - \beta_2}{\beta_1-1} $, we then need to show
\begin{equation}
\norm{\left(I - \gamma V_mV_m^T \right)x}_2 < \norm{x}_2.
\label{blah}
\end{equation}
Note that we have that $V_mV_m^T x \neq 0$ from the assumptions. This is because, 
$$
V_mV_m^T x = \left(V_mV_m^T - \beta_1 V_mV_m^T \right) \frac{df(t)}{dt},
$$
since $ V_mV_m^T  V_mV_m^T \frac{df(t)}{dt} =  V_mV_m^T \frac{df(t)}{dt}$, as $V_mV_m^T \frac{df(t)}{dt}$ is already entirely within $\mathcal{K}$. Then,
$$
V_mV_m^T x  = (1-\beta_1) V_mV_m^T \frac{df(t)}{dt}.
$$
We have that $1-\beta_1 = \frac{1}{S_1} \neq 0$ and $ V_mV_m^T \frac{df(t)}{dt} \neq 0$, so that $V_mV_m^T x  \neq 0$. \\
To prove the lemma we apply (\ref{oProjGen}) to $x$, with $\gamma$ in place of $\alpha$. If we have that $[(1-\gamma)^2 -1] <0$, then \eqref{blah} is true since $V_mV_m^T x \neq 0$. An equivalent requirement is $\gamma \in (0,2)$. Some algebra gives us $\gamma = 1 - \frac{S_1}{S_2}$. Since  $S_2>S_1$, it follows that $\gamma \in (0,2)$.
\end{proof}
From Lemma (\ref{anotherlemma}) we see that $S$ recycled Krylov substeps not only maintains the local error order of the ETD1 scheme, but also decreases the $2$-norm of the leading term with increasing $S$. Note that the leading term does not tend towards zero as $S \rightarrow \infty$, but towards a constant. We thus expect diminishing returns in the increase in accuracy with increasing $S$, and the existence of an optimal $S$ for efficiency.

\section{Using the additional substeps for correctors}
\label{cor-sec}
We now establish a new second order scheme based on a 
finite difference approximation to the derivative of the nonlinear
term $F$ and the recycling scheme given in (\ref{recstart}), (\ref{recnext}).

The first step is to expand the local error for the standard ETD1 scheme.
Using variation of constants and a Taylor series expansion of $F(t,
u(t))$, the exact solution of (\ref{semilin}) can be expressed as a
power series (see for example \cite{Overview,RKEI})
\begin{equation}
u(t_n+\Delta t) = e^{\Delta t L}u(t_n) + \sum_{k=1}^{k}{\Delta t^k \varphi_k(\Delta t L) F^{(k-1)}(t_n, u_n)} + O(\Delta t^k),
\label{true1}
\end{equation}
with $F^{(k)}(t_n, u_n) = \frac{d^k F}{dt^k}(t_n, u_n)$. Under the local error assumption $u_n = u(t_n)$, the local error of the ETD1 step given in (\ref{etd1}) is
\begin{equation}
E_{n+1}^{etd} \equiv u(t_n+\Delta t)  - u_n^{etd} +  p_{\Delta t}g_n = \sum_{k=2}^{\infty}{\Delta t^k\varphi_k(\Delta t L)F^{(k-1)}(t_n)}.
\label{trueerror}
\end{equation}
Since the approximation from a substepping scheme is related to the approximation from the ETD1 scheme (over one step) by $u_{n+1} = u_{n+1}^{etd} + R_{n+1}^S$, we have the local error for the recycling scheme:
\begin{equation}
u(t_n+ \Delta t) - u_n  = E_{n+1}^{etd}  - \mathcal{E}_{n+1}^m - R_{n+1}^S.
\label{thing2}
\end{equation}
The terms of error expression (\ref{thing2}) at arbitrary order can be
found using \eqref{true1}, Lemma \ref{lemmaForR}, and the information
on Krylov projection methods in \secref{sec-kry-proj}. We see that the
expansion consists of terms involving the value of $F$ or derivatives
thereof at various substeps. These terms can be approximated by finite
differences of the values for $F$ at the different substeps, and used
as a corrector to eliminate terms for the error. 

We consider extrapolation in the leading error in the case of two substeps, that is $S \equiv 2$. Assume that the error from the Krylov approximation, $\mathcal{E}_n^m$ is negligible compared to $E_n$ and $R_n$, so that it does not introduce any terms at the first and second and third order expansion of $E_n$ and $R_n$. Then we can express exactly the leading second and third order error terms. \\
First we have the leading terms of $E_n^{etd}$ from (\ref{trueerror}),
\begin{equation}
\begin{split}
E_n^{etd} &= \Delta t^2 \varphi_2(\Delta t L)F^{(1)}(t_n,u_n) + \Delta t^3 \varphi_3(\Delta t L)F^{(2)}(t_n,u_n) + O(\Delta t^4) \\
    &= \frac{\Delta t^2}{2!}F^{(1)} +  \frac{\Delta t^3 L}{3!}F^{(1)} +  \frac{\Delta t^3}{3!}F^{(2)} + O(\Delta t^4).
\end{split}
\end{equation}
We also have the leading terms of $R_{n+1}^2$ (from two substeps, recall \eqref{Rdef})
\begin{equation}
\begin{split}
R_{n+1}^2 &= \tilde{p}_{\frac{\Delta t}{2}}( F_{n+\frac{1}{2}}- F_n ) \\
          &= \frac{\Delta t}{2} V_m \left( I + \frac{\Delta t H_m}{2} + \ldots \right) V_m^T  ( F_{n+\frac{1}{2}}- F_n ) \\
          &= \frac{\Delta t}{2} V_m  V_m^T  ( F_{n+\frac{1}{2}}- F_n ) + V_m\frac{\Delta t^2 H_m}{4}V_m^T  ( F_{n+\frac{1}{2}}- F_n ) + O(\Delta t^3).
\end{split}
\label{R_np1_eq}
\end{equation}
Note that the terms in \eqref{R_np1_eq} are an order higher than written since $F_{n+\frac{1}{2}}- F_n = \frac{\Delta t}{2} F^{(1)}(t_n,u_n) + O(\Delta t^2)$. We then have that
\begin{equation}
\begin{split}
u(t_n+\Delta t) &= u_{n+1}  
+ \frac{\Delta t^2}{2!}F^{(1)} - \frac{\Delta t}{2} V_m  V_m^T  ( F_{n+\frac{1}{2}}- F_n ) \\
								&+  \frac{\Delta t^3 L}{3!}F^{(1)}+ \frac{\Delta t^3}{3!}F^{(2)} -  V_m\frac{\Delta t^2 H_m}{4}V_m^T  ( F_{n+\frac{1}{2}}- F_n ) + O(\Delta t^4).
\end{split}
\label{leading}
\end{equation}
The idea now is as follows. Define a corrected approximation:
\begin{equation}
u_{n+1}^{(c)} \equiv u_{n+1} + C - \frac{\Delta t}{2} V_m  V_m^T  ( F_{n+\frac{1}{2}}- F_n ).
\label{corthis}
\end{equation}
In (\ref{corthis}), $C$ is a corrector intended to cancel out some of the leading terms in (\ref{leading}). The term $\frac{\Delta t}{2} V_m  V_m^T  ( F_{n+\frac{1}{2}}- F_n )$ is the only leading term in (\ref{leading}) to involve the matrix $V_m$, and so is added directly to the corrected approximation (\ref{corthis}) to allow $C$ to be free of dependence on the matrix $V_m$. Indeed, $C$ will be a linear combination of the the three function values of $F$, $F_n$, $F_{n+\frac{1}{2}}$ and $F_{n+1}$, available at the end of the full step. The approximation to $u$ produced by substeps of the scheme, and thus also to $F$, is locally second order. We define the $C$ term as follows, with coefficients $\alpha$, $\beta$, $\gamma$ to be chosen later.

\begin{equation}
\begin{split}
C  &\equiv \Delta t \alpha F_n + \Delta t \beta F_{n+\frac{1}{2}} + \Delta t \gamma F_{n+1} \\
    &= \Delta t \alpha F(t_n) + \Delta t \beta F\left(t_n +\frac{\Delta t}{2}\right) + \Delta t \gamma F\left(t_n + \Delta t\right) + \Delta t^3 E_c + O(\Delta t ^4) \\
    &= \Delta t \left(\alpha + \beta + \gamma \right) F + \Delta t^2 \left(\frac{\beta}{2} + \gamma \right) F^{(1)} + \frac{\Delta t ^3}{2}\left(\frac{\beta}{4} + \gamma \right) F^{(2)} + \Delta t^3E_c + O(\Delta t ^4)
\end{split}
\label{corrector1}
\end{equation}
where we have used that $F_n = F(t_n, u(t_n))$ (under the local error assumptions), $F_{n+\frac{1}{2}}= F\left(t_n +\frac{\Delta t}{2}\right) +O(\Delta t^2)$ and so on. The new term $ \Delta t^3 E_c$ is introduced to represent the $O(\Delta t^3)$ error in writing $\Delta t F_{n+\frac{1}{2}}$ as $\Delta t F\left(t_n +\frac{\Delta t}{2}\right)$, and so on. \\
From (\ref{corrector1}), we must choose the coefficients to satisfy the two conditions
$$
\alpha + \beta + \gamma = 0, \qquad \text{and } \qquad 
\frac{\beta}{2} + \gamma = \frac{1}{2}.
$$
With these values of the parameters, the local error of the corrected approximation is 
\begin{equation}
\begin{split}
& u(t_n+\Delta t) - u_{n+1}^{(c)} =   \\
								&  \frac{\Delta t^3}{3!}F^{(2)} -  V_m\frac{\Delta t^2 H_m}{4}V_m^T  ( F_{n+\frac{1}{2}}- F_n ) + \frac{\Delta t ^3}{2}\left(\frac{\beta}{4} + \gamma \right) F^{(2)} -\Delta t^3E_c + O(\Delta t^4) \\
								& =  \frac{\Delta t^3}{3!}F^{(2)} -  V_m\frac{\Delta t^2 H_m}{8}V_m^T F^{(1)} + \frac{\Delta t ^3}{2}\left(\frac{\beta}{4} + \gamma \right) F^{(2)} -\Delta t^3E_c + O(\Delta t^4) .
\end{split}
\label{leading2}
\end{equation}
We have three coefficients to determine, and two constraints. We are therefore in a position
to pick another constraint to reduce the new leading error in (\ref{leading2}). It would be helpful to know the form of the error term $E_c$, introduced by the approximation of $F$ in (\ref{corrector1}). We have: 
\begin{equation}
\begin{split}
F_{n+\frac{1}{2}} &= F \left( t_{n+\frac{1}{2}}, u\left(t_{n+\frac{1}{2}}\right) -\frac{\Delta t^2}{8}F' +O(\Delta t^3) \right) \\
                  &= F \left( t_{n+\frac{1}{2}}, u\left(t_{n+\frac{1}{2}}\right)  \right) - \frac{\Delta t^2}{8}\frac{\partial F}{\partial u}F' + O(\Delta t^3),
\end{split}       
\end{equation}
using Corollary \ref{finalTayErr}.
We also have
\begin{equation}
\begin{split}
F_{n+\frac{2}{2}} &= F \left( t_{n+1}, u\left(t_{n+1}\right) - \frac{\Delta t^2}{2} \left( I - \frac{1}{2}V_mV_m^T  \right)\frac{dF(t)}{dt} +O(\Delta t^3) \right) \\
                  &= F \left( t_{n+1}, u\left(t_{n+1}\right)  \right) - \frac{\Delta t^2}{2}\frac{\partial F}{\partial u}\left( I - \frac{1}{2}V_mV_m^T  \right)\frac{dF(t_n)}{dt} + O(\Delta t^3),
\end{split}       
\end{equation}
$E_c$ is then 
$$
-\beta \frac{1}{8}\frac{\partial F}{\partial u}\frac{dF(t_n)}{dt}- \gamma \frac{1}{2}\frac{\partial F}{\partial u}\left( I - \frac{1}{2}V_mV_m^T  \right)\frac{dF(t_n)}{dt}.
$$
Substituting into (\ref{leading2}) ,
\begin{equation}
\begin{split}
 u(t_n+\Delta t) & - u_{n+1}^{(c)}  =  \\
								& \frac{\Delta t^3}{3!}F^{(2)} -  \frac{\Delta t^2 }{4}V_m H_m V_m^T  ( F_{n+\frac{1}{2}}- F_n ) - \frac{\Delta t ^3}{2}\left(\frac{\beta}{4} + \gamma \right) F^{(2)}   \\
								& - \Delta t^3 \frac{\partial F}{\partial u} \left( \left(\frac{\beta}{8} + \frac{\gamma}{2}\right)I - \frac{\gamma}{4} V_mV_m^T \right) \frac{dF(t_n)}{dt}.
\end{split}
\label{leading3}
\end{equation}
We have the option here to use the final constraint to eliminate the coefficient of $F^{(2)}$ in the leading term:
$$
\frac{\Delta t^3}{3!} - \frac{\Delta t ^3}{2}\left(\frac{\beta}{4} + \gamma \right)   = 0.
$$
Note that $E_c$ cannot be eliminated without taking the inverse of $V_mV_m^T$, so this is not an efficient option. It can be seen that the values that satisfy the three constraints are:
$$
\alpha = -\frac{5}{6}, \hspace{2mm} \beta = \frac{2}{3}, \hspace{2mm} \gamma = \frac{1}{6}.
$$
Of course $E_c$ also depends on the values of $\alpha$, $\beta$,
$\gamma$, so the magnitude of the third order term will be affected by
the choice of these values also through $E_c$. With the choices given
above, we have the numerical scheme
\begin{equation}
\label{eq:corrector}
u_{n+1}^{(c)} =  u_{n+1} + C  - \frac{\Delta t}{2} V_m  V_m^T  ( F_{n+\frac{1}{2}}- F_n );
\end{equation}
that is,
\begin{equation}
\begin{split}
u_{n+1}^{(c)} =   u_{n+1} &\\
& - \Delta t \frac{5}{6} F_n + \Delta t \frac{2}{3} F_{n+\frac{1}{2}} + \Delta t  \frac{1}{6} F_{n+1} \\
& - \frac{\Delta t}{2} V_m  V_m^T  ( F_{n+\frac{1}{2}}- F_n ).
\end{split}
\end{equation}
and the $E_c$ term in (\ref{leading3}) becomes:
$$
- \Delta t^3 \frac{\partial F}{\partial u} \left( \frac{2}{12}I - \frac{1}{24} V_mV_m^T \right) \frac{dF(t_n)}{dt}.
$$
Here we have used all the extra information from the two substeps to
completely eliminate the lowest order from the local error, and a part
of the new leading order term for the scheme. A more thorough use of
the error expressions in the lemmas here may give rise to recycling
schemes that use more substeps and are able to completely eliminate
higher order terms from the error, leading to a kind of new
exponential Runge-Kutta framework involving recycled Krylov
subspaces. Below we demonstrate the efficacy of our two-step corrected
recycling scheme with numerical examples. In 
\ref{App:EEM} we show how to apply the analysis of the substepping method to the
locally third order exponential integrator scheme EEM.

\section{Numerical Results}
\label{sec:numres}
Here we examine the performance of the recyling scheme
(\ref{recstartsuccinct}), (\ref{recnextsuccinct}) and the corrector
scheme \eqref{eq:corrector} (for the first two examples). 
The PDEs investigated in these experiments are all
advection-diffusion-reaction equations, which are converted into
semilinear ODE systems \eqref{semilin} by spatial disretisation before
our timestepping schemes are applied, see for example, \cite{myThesis} for
more details. We use the $v(\mathbf{x}, t) \in \mathcal{R}$ to
represent the solution to the PDE, while $u(t) \in \mathcal{R}^N$
represents the solution of the corresponding ODE system. 
The spatial discretisation is a simple finite volume method in
all the examples. In examples 2 and 3, the grid was using code from
MRST \cite{mrst_prime}.
We compare the second order corrector scheme \eqref{eq:corrector} to
both the standard second order exponential integrator (ETD2; refer to for example Equation (6) in \cite{cm}) and
standard second order Rosenbrock scheme (ROS2; the same as the simplest Rosenbrock scheme described in  \cite{rosenbrock_pre, rosenbrock}). For the first two experiments, the error is estimated by comparison
with a low $\Delta t$ comparison solve $u_{\mbox{comp}}$ with
ETD2. Our ETD2 implementation uses phipm.m \cite{NW} for each
timestep; which requires the following parameters: an initial Krylov
subspace dimension $m$, and an error tolerance. For our comparison solve
runs, we used $m=30$ and $10^{-7}$ respectively for these
parameters. The first two experiments are also found in
\cite{myThesis}; see this for more details.  
For the third experiment a comparison solve was prohibitively time
consuming, so error was instead estimated by differencing successive
results.
We estimate the error in a disrecte aproximation of the $L^2(\Omega)$ norm,
where $D$ is the computational domain.

\subsection{Allen-Cahn type Reaction Diffusion}
\label{AC-test-kryrec}
We approximate the solution to the PDE,
$$
\frac{dv}{dt} =  \nabla^2 D v + v - v^3.
$$
The (1D) spatial domain for this experiment was $\Omega = [0,100]$
This was discretised into a grid of $N=100$ cells. We imposed no flow
boundary conditions, i.e., $\frac{\partial u}{\partial x} = 0$ where
$x = 0$ or $x = 100$. There was a uniform diffusivity field of $D(x) = 1.0$. 
The initial condition was $u(x,0) = \cos \left( \frac{2 \pi x}{N}
\right)$ and we solved to a final time $T=1.0$. 

In \figref{AC1-steps} a) and c) we show the estimated error against
$\Delta t$, for the recycling scheme with varying number of substeps
$S$, ($S=1,2,5,10,50,100$). Note that $S=1$ is the standard ETD1 integtrator.
The behaviour is as expected; increasing
$S$ decreases the error and the scheme is first order. The diminishing returns of increasing $S$ (see \eqref{leadrecdt}) can also be observed; for example compare the significant increase in accuracy in increasing $S$ from $1$ to $5$, with the lesser increase in accuracy in increasing $S$ from $5$ to $10$. \figref{AC1-steps} shows this more emphatically - the increase in accuracy in increasing $S$ from $10$ to $50$ is significant, but the effect of increasing $S$ from $50$ to $100$ is very small. The limiting value of the error with respect to $S$ discussed above is clearly close to being reached here. \\ 
In \figref{AC1-steps} b) and d) we plot estimated error against
cputime to demonstrate the efficiency of the scheme with varying
$S$. In this case increasing $S$ appears to increase efficiency until
an optimal $S$ is reached, after which it decreases, as predicted. 
\figref{AC1-steps} d) shows that the optimal $S$ lies between $50$ and $100$ for this system. \\
In \figref{AC1-O2} we examine the 2-step corrector
\eqref{eq:corrector}.
Plot a) shows estimated error against $\Delta t$. The corrector scheme
is second order as intended, and has quite high accuracy compared to
the other two schemes, possibly due to the heuristic attempt to
decrease the error in the leading term (see discussion in
\secref{cor-sec}). In plot b) we see that the 2-step corrector is of
comparable efficiency to ROS2. 

In \figref{AC1-steps} b) we see that for the same cputime, increasing $S$ from $1$ to $10$ decreases the estimated error by roughly an order of magnitude. We can see in \figref{AC1-steps} d) that increasing $S$ from $10$ to $50$ can further decrease error for a fixed cputime, though less significantly. Comparing a fixed cputime in \figref{AC1-steps} b) and \figref{AC1-O2} b) indicates that the second order, 2-step corrector method can produce error more than one order of magnitude smaller than the first order recycling scheme with $S=10$. 

\subsection{Fracture system with Langmuir-type reaction}
\label{fracreac_kryrec}
We approximate the solution to the PDE,
\begin{equation}
\frac{dv}{dt} =  \nabla \cdot (  \nabla D v + V v  ) - \frac{0.02}{D(\mathbf{x})^2} \frac{v}{1+v}.
\label{lang_type_adr}
\end{equation}
where $D(\mathbf{x})$ is the diffusivity and $V(\mathbf{x})$ is the
velocity. In this example a single layer of cells is used, making the
problem effectively two dimensional. The domain is $\Omega = 10 \times
10 \times 10$ metres, divided into $100 \times 100 \times 1$ cells of
equal size. We impose no-flow boundary conditions on every edge. The initial condition imposed is initial $v(\mathbf{x})=0$ everywhere except at $\mathbf{x} = (4.95, 9.95)^T$ where $v(\mathbf{x})=1$.
The diffusivity $D$ in the grid varies with $\mathbf{x}$, in a way
intended to model a fracture in the medium. A subset of the cells in
the 2D grid were chosen to be cells in the fracture. These cells were
chosen by a weighted random walk through the grid (weighted to favour
moving in the positive $y$-direction so that the fracture would bisect
the domain). This process started on an initial cell which was marked
as being in the fracture, then randomly chose a neighbour of the cell
and repeated the process. We set the diffusivity to be $D = 100$ on
the fracture and $D=0.1$ elsewhere. There is also a constant velocity
$V$ field in the system, uniformly one in the x-direction and zero in
the other directions in the domain, i.e., $\mathbf{v}(\mathbf{x}) =
(1,0,0)^T$ , to the right in \figref{frac-lang-show}. The initial
condition was $c(\mathbf{x})=0$ everywhere except at $\mathbf{x} =
(4.95, 9.95)^T$ where $c(\mathbf{x})=1$. 

In \figref{frac-lang-show} we show the final state of the system at time $T=2.4$. The result in plot a) was produced with the 2-step recycling scheme with a timestep $\Delta t = 2.4 \times 10^{-4}$. Plot b) shows the high accuracy comparison ETD2 solve, produced with $\Delta t = 2.4 \times 10^{-5}$. \\
In \figref{frac-lang-steps} we demonstrate the effect of increasing
the number of substeps $S$  on the error. \figref{frac-lang-steps} a)
shows estimated error against timestep $\Delta t$, for schemes using
$S=1,2,5,10$ substeps, while \figref{frac-lang-steps} c) shows the
same for schemes using $S=10, 50,100$. Recall that $S=1$ is the
standard ETD1 integtrator. 

For sufficiently low $\Delta t$ we have the predicted results, with the error being first order with respect to $\Delta t$, and decreasing as $S$ increases. For $\Delta t$ too large, this is not the case. Here the Krylov subspace dimension $m$ is most likely the limiting factor as Assumption \ref{asskryerr} becomes invalid. In \figref{frac-lang-steps} b) and d) we show the efficiency by plotting the estimated error against cputime. For $\Delta t$ low enough that the substepping schemes are effective, the scheme with $10$ substeps is the most efficient.\\
We can see the existence of an optimal $S$ for efficiency, as predicted, in \figref{frac-lang-steps} d), where the scheme using $S=50$ is more efficient than the scheme using $S=100$. Any increase in accuracy by increasing $S$ from $50$ to $100$ is extremely small (indeed, it is unnoticeable in \figref{frac-lang-steps} c), and not enough to offset the increase in cputime. In fact, \figref{frac-lang-steps} d) shows that for this experiment the scheme using $S=10$ is more efficient than both the $S=50$ and $S=100$ schemes. \figref{frac-lang-steps} c) shows that the $S=10$ scheme is also slightly more accurate than both. This is likely because at $S=10$ the improvement in accuracy is already close to the limiting value, and greatly increasing $S$ to $50$ or $100$ only accumulates rounding errors without further benefit. \figref{frac-lang-steps} a) shows that the improvement from $S=1$ to $S=10$ is quite significant on its own.  

In \figref{frac-lang-cor} we compare the 2-step corrector scheme
against the two other second order exponential integrators, ETD2 and ROS2. \figref{frac-lang-cor} a) shows estimated error against $\Delta t$, and we see that, like \figref{frac-lang-steps} a), the Krylov recycling scheme does not function as intended above a certain $\Delta t$ threshold; again this is due to the timestep being too large with respect to $m$. The standard exponential integrators do not have this problem, as their timesteps are driven by phipm.m, which takes extra (non-recycled, linear) substeps to achieve a desired error. Below the $\Delta t$ threshold, the 2-step corrector scheme functions exactly as intended, exhibiting second order convergence and high accuracy. In \figref{frac-lang-cor} b) we can observe that the 2-step corrector scheme is more efficient than the other two schemes for lower $\Delta t$, and of comparable efficiency for larger $\Delta t$. \\
It is interesting to compare \figref{frac-lang-steps} a) and \figref{frac-lang-cor} a) and note that the threshold $\Delta t$ for the corrector scheme seems to be lower than for the substepping schemes. \\
In \figref{frac-lang-steps} b) we can again see that for a fixed cputime, increasing $S$ from $1$ to $10$ decreases error by roughly one order of magnitude; however \figref{frac-lang-steps} d) shows no improvement in increasing $S$ from $10$ to $50$. Comparing \figref{frac-lang-steps} b) and \figref{frac-lang-cor} b) shows that the second order corrector scheme can be almost three orders of magnitude more accurate for a fixed cputime than the first order recycling scheme with $S=10$.

\subsection{Large 2D Example with Random Fields}
\label{large}
In this example the 2D grid models a domain with physical dimensions
$100 \times 100 \times 10$; the grid is split into a $1000 \times 1000
\times 1$ cells. The model equation is the same as the previous
example \eqref{lang_type_adr}. The diffusivity is kept constant at
$D=0.01$, while a random velocity $V$ field is used. For this we
generated a random permativity field $K$, which was then used to
generate a corresponding pressure field and then a velocity field in a
standard way, using Darcy's Law, see \cite{myThesis,mrst_prime}. The pressure $p$ field was determined by the permativity field and the Dirichlet boundary conditions $p=1$ where $y=0$ and $p=0$ where $y=100$. 
The initial conditions for $v$ were zero everywhere, and the boundary
conditions were the same as for $p$,  $v=1$ where $y=0$ and $v=0$
where $y=100$. The final time was $T=500$. 

Due to the large size of system ($10^6$ unknowns) we only examine the
recycling scheme and for a system of this size, it was necessary to
increase $m$ to $100$ to prevent the Krylov error from being
dominant. The results are shown in Figure \ref{giant-results}. We see
that, for $\Delta t$ sufficiently low, increasing $S$ decreases the
error and increases the efficiency of the scheme. The improvement in
efficiency between $S=5$ and $S=10$ is marginal; the optimal $S$ for
this example would not be much greater than $10$. 

\section{Conclusions}
We have extended the notion of recycling a Krylov subspace for
increased accuracy in the sense of \cite{2steppap}. We have applied
this new method to the first order ETD1 scheme and examined the effect of
taking an arbitrary number $S$ of substeps. The local error has been
expressed in terms of $S$, and the expression shows that the local
error will decrease with $S$ down to a finite limit. 
The discussion in \ref{App:EEM} examines construction for EEM.
Results suggest that there maybe an optimal $S$ for a maximal
efficiency increase and 
some preliminary analysis in this direction may be fuond in \cite{myThesis}.
Convergence and existence of an optimal $S>1$ has been demonstrated
with numerical experiments.  
Additional information from the substeps was used to form a corrector
and a second order scheme. This was shown to be comparable to, or slightly
better than, ETD2 and ROS2 in our tests. 

The schemes currently rely on Assumption \ref{asskryerr}, essentially requiring that $\Delta t$ be sufficiently small and $m$ be sufficiently large, to be effective. Numerical experiments have shown how having $\Delta t$ too large can cause the schemes to become inaccurate as the error of the initial Krylov approximation becomes significant. It is already well established how the Krylov approximation error can be controlled by adapting $m$ and the use of non-recycling substeps. Applying these techniques to the schemes presented here in future work would allow them to be effective over wider $\Delta t$ ranges. 

\section{Acknowledgments}
The authors are grateful to Prof. S. Geiger for his input into the
flow simulations. The work of the Dr D. Stone was funded by the
SFC/EPSRC(EP/G036136/1) as part of NAIS.

\bibliographystyle{unsrt}
\bibliography{Bibliography_initials}

\appendix
\section{Substepping with the scheme EEM}
\label{App:EEM}
We now show how to apply the analysis of the substepping method to the
locally third order exponential integrator scheme EEM. Applied to the
system of ODEs 
$$ \frac{du}{dt} = g(u),$$
where $g(u)$ may not be semilinear, the scheme EEM is given by
\begin{equation}
u_{n+1} = u_n + \Delta t \phi_1(\Delta t J_n),
\label{eq:EEM}
\end{equation}
where $J$ denotes the Jacobian of $g$ and $J_n = J(u_n)$. 
The Jacobian $J_n$ is kept fixed for the entire step $\Delta t$, including
recycling substeps. Therefore an $S$ step recycling scheme can be
defined on EEM in exactly the same way as the 
recycling scheme for ETD1. Note that the Krylov subspace will be
generated for $J$ and $g$ in the EEM case, i.e. $\mathcal{K} =
\mathcal{K}(J_n, g_n)$. 

With $\tilde{p}_{\tau} \equiv \tau V_m \varphi_1(\tau H_m) V_m^T$
approximating $\tau \varphi_1(\tau J)$
Applying the Krylov subspace recycling scheme to EEM we have 
$$
u_{n+\frac{j}{S}}=u_n +  \tilde{p}_{j\delta t}g_n + R_{n+\frac{j}{S}}^S.
$$
Following the same steps as in Lemma \ref{lemmaForR} we obtain the
following result.
\begin{corollary}
The remainder $R_{n+\frac{j}{S}}^S$ satisfies the recursion relation
\begin{equation}
R_{n+\frac{j+1}{S}}^S = \tilde{p}_{\delta t}\DkF{j} + (I+\tilde{p}_{\delta t}L)R_{n+\frac{j}{S}}^S,
\label{R_ind}
\end{equation}
where $F_{n+\frac{j+1}{S}} = g_{n+\frac{j+1}{S}} - J_n u_{n+\frac{j+1}{S}}$.
\end{corollary}
To examine the remainder term in more detail let $\hat{J}_i$ be the Hessian matrix 
$$
\hat{J}_i = \left(
\begin{array}{ccc}
(\hat{g}_i)_{x_1 x_1} & (\hat{g}_i)_{x_1 x_2} & \ldots \\
(\hat{g}_i)_{x_2 x_1} & (\hat{g}_i)_{x_2 x_2} & \ldots \\
\ldots & \ldots & \ldots 
\end{array}
\right),
$$
where $\hat{g}_i$ is the $i$th entry of the vector $g$.
Let the tensor $\hat{\mathbf{J}}$ be a vector with the matrix
$\hat{J}_i$ in its $i$th entry. 
We can now Taylor expand the remainder $ R_{n+\frac{j}{S}}^S $ from
\eqref{R_ind} to find the local error of the EEM scheme with recycled substeps. 
\begin{lemma}
For the EEM recycling scheme, the leading term of $ R_{n+\frac{j}{S}}^S $ satisfies
$$
R_{n+\frac{j}{S}}^S = \alpha(j) \delta t^3 V_mV_m^T g_n^T \hat{\mathbf{J}} g_n + O(\delta t^4)
$$
where $\alpha (j)=(2j^3 - 3j^2 + j)/24$.
\end{lemma}
\begin{proof}
By induction. The base case is true for $j=1$  with $\alpha(1) = 0$ since there is no recycling at that step. Assume true for some $j$. Consider $g_{n+\frac{j}{S}}$,
$$
g_{n+\frac{j}{S}} = g(u_{n+\frac{j}{S}}) = g(u(t) + j\delta t g(t) + \frac{1}{2}(j \delta t)^2 J g + O(\delta t^3) ),
$$
to second order this is
$$
g_{n+\frac{j}{S}} = g_n+ J \left(   j\delta t g_n + \frac{1}{2}(j \delta t)^2 J g_n \right) + \frac{(j\delta t)^2}{2} g_n^T \hat{\mathbf{J}}  g_n+ O(\delta t^3),
$$
where we have made use of the local error assumption $u(t_n) = u_n$.
Then,
$$F_{n+\frac{j}{S}}-F_n = g_{n+\frac{j}{S}} - g_n - J
u_{n+\frac{j}{S}} + J u_n $$
and since $u_{n+\frac{j}{S}} = u_n  j\delta t g_n + \frac{(j \delta
  t)^2}{2} J g_n+ O(\delta t^3) $ up to second order and the induction
hypothesis we have 
$$
F_{n+\frac{j}{S}}-F_n = \frac{(j\delta t)^2}{2} g_n^T \hat{\mathbf{J}}  g_n + O(\delta t^3).
$$
Now consider
$$
\tilde{p}_{\delta t}\left( F_{n+\frac{j}{S}}-F_n  \right) = \frac{j^2(\delta t)^3}{2} V_mV_m^Tg_n^T \hat{\mathbf{J}}  g_n+  O(\delta t^4) 
$$
The induction relation for $R_{n+\frac{j+1}{S}}^S $ then gives us
$$
R_{n+\frac{j+1}{S}}^S =  \frac{j^2(\delta t)^3}{2} V_mV_m^T g_n^T \hat{\mathbf{J}}  g_n + (I+\tilde{p}_{\delta t}J)R_{n+\frac{j}{S}}^S +  O(\delta t^4) 
$$
which to leading order this is
$$
R_{n+\frac{j+1}{S}}^S = \left( \frac{j^2}{2} +\alpha(j) \right) \delta t^3 V_mV_m^T  g_n^T \hat{\mathbf{J}} g_n +  O(\delta t^4).
$$
So $\alpha(j+1) = \frac{j^2}{2} + \alpha(j)$, $\alpha(1) = 0.$
which is satisfied by the given $\alpha(j)$. 
\end{proof}
We now combine the leading term of the remainder $R$ and the known
local error of EEM 
$$
\frac{1}{6} \Delta t^3 g^T \hat{\mathbf{J}} g.
$$
(see, for example, \cite{myThesis}) to find the local error of the new recycling scheme.
\begin{corollary}
The leading term of the local error of the $S$ step recycling scheme for EEM at the end of a timestep is
$$
\frac{\Delta t^3}{6} \left(  I    - \left( \frac{2S^2 - 3S + 1}{2 S^2} \right)  V_mV_m^T  \right)  g^T \hat{\mathbf{J}}  g.
$$
\end{corollary}
From this we can predict similar properties to the ETD1 recycling
scheme. This extends the work of \cite{2steppap}, where the recycling substepping EEM scheme was used for a single substep.

\begin{figure}[h]
\centering
\begin{minipage}[b]{0.49\linewidth}
(a) \\
\includegraphics[width=0.99\columnwidth]{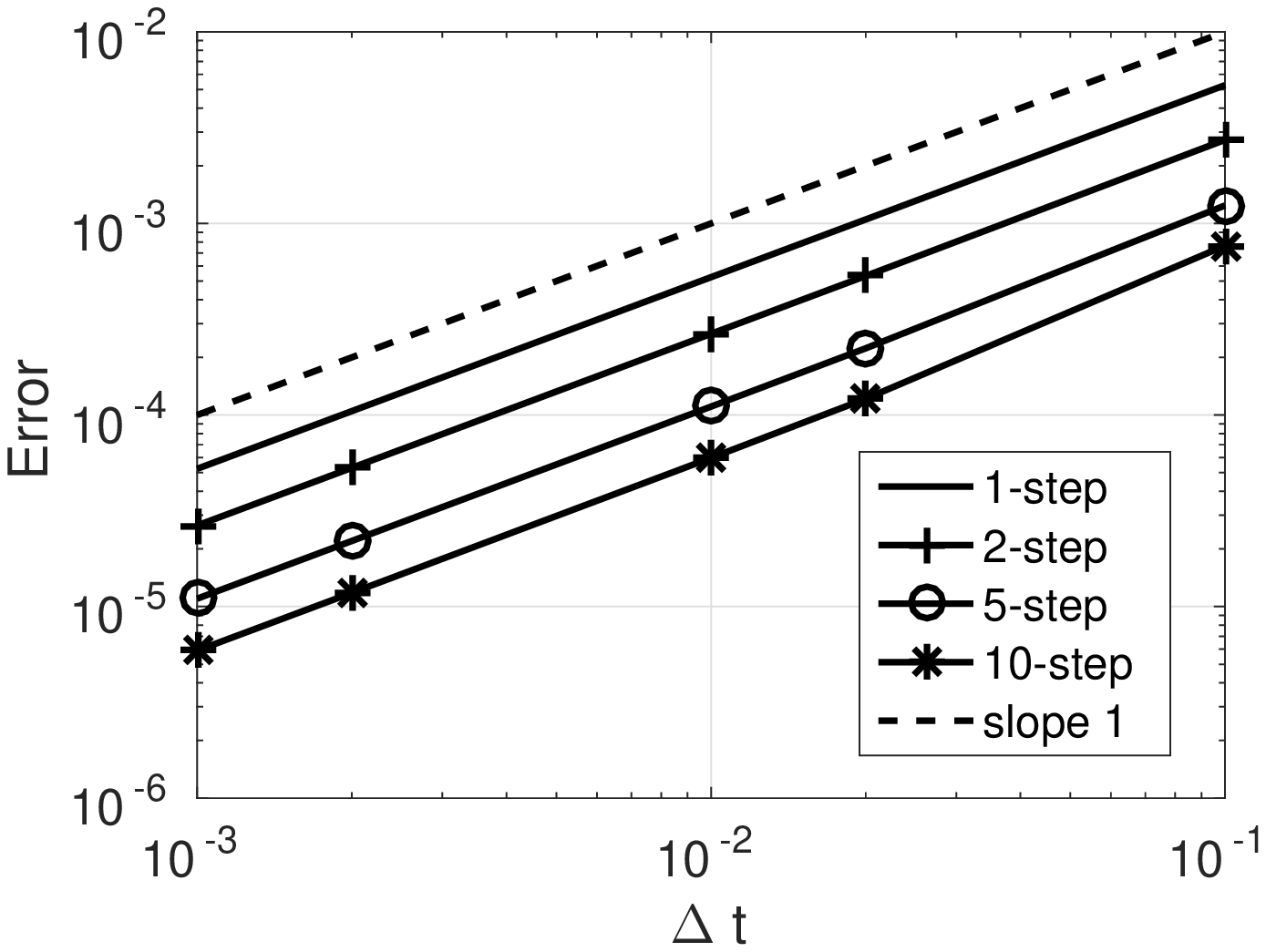}
\end{minipage}
\begin{minipage}[b]{0.49\linewidth}
(b) \\
\includegraphics[width=0.99\columnwidth]{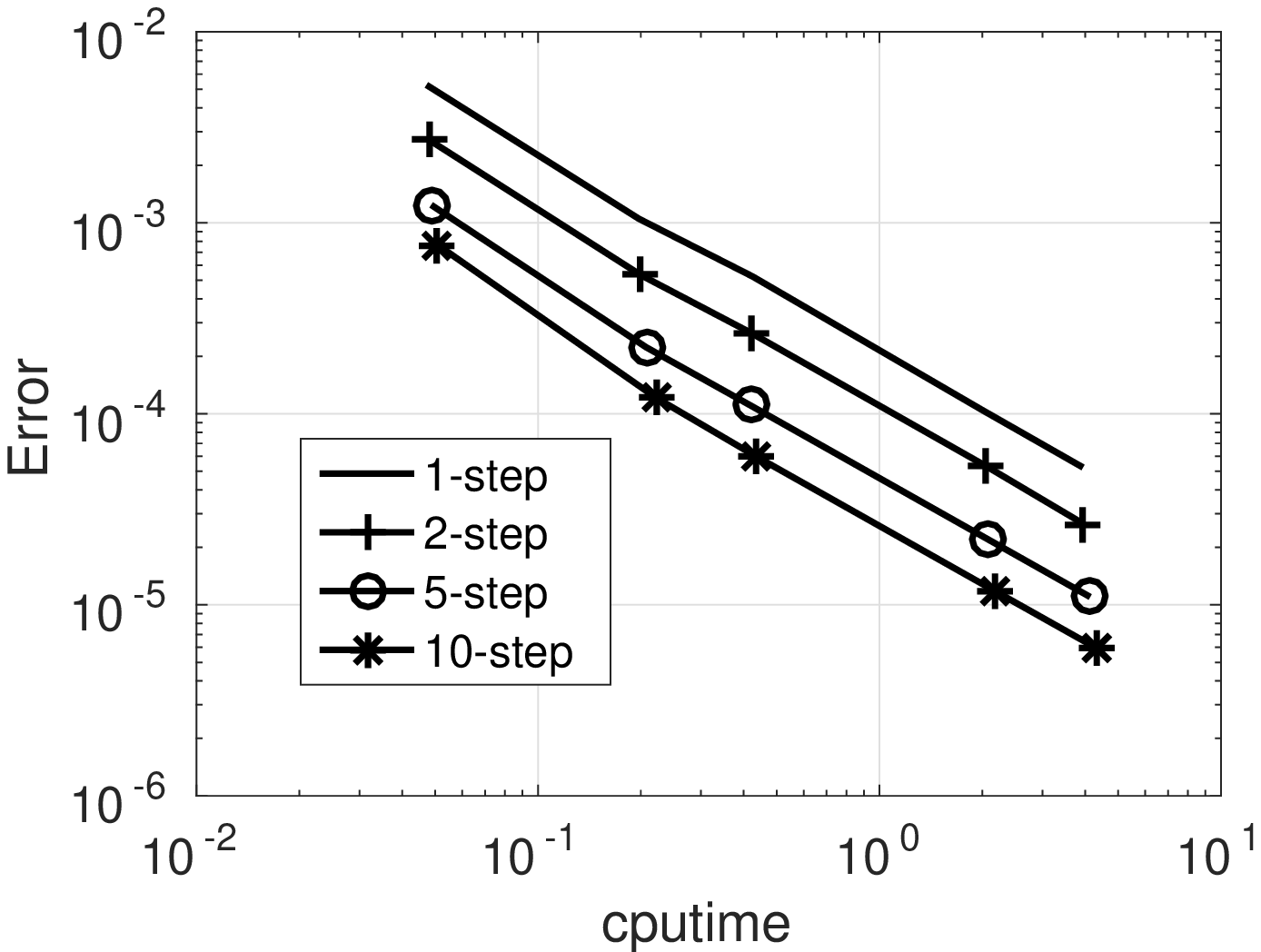}
\end{minipage}
\begin{minipage}[b]{0.49\linewidth}
(c) \\
\includegraphics[width=0.99\columnwidth]{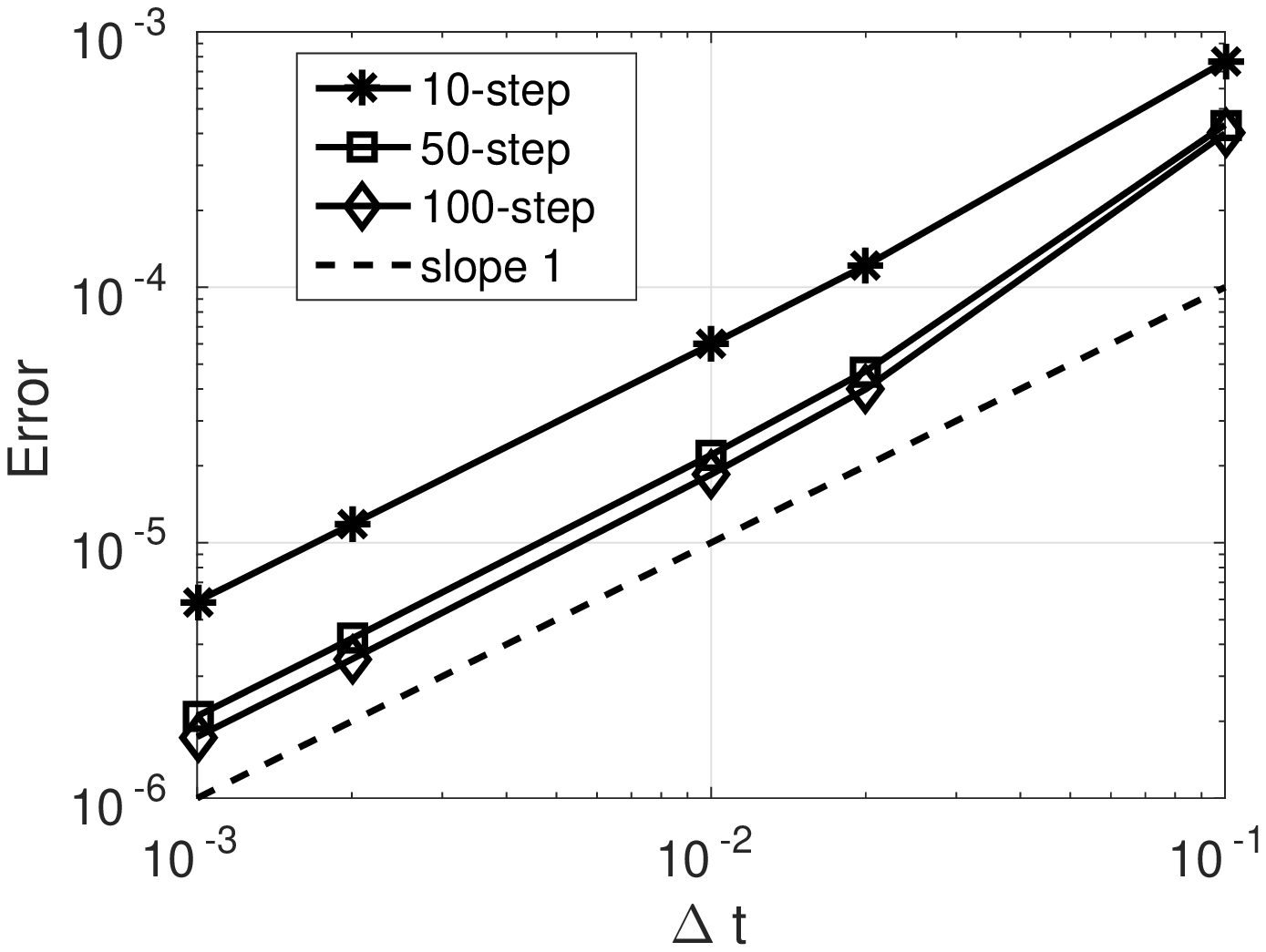}
\end{minipage}
\begin{minipage}[b]{0.49\linewidth}
(d) \\
\includegraphics[width=0.99\columnwidth]{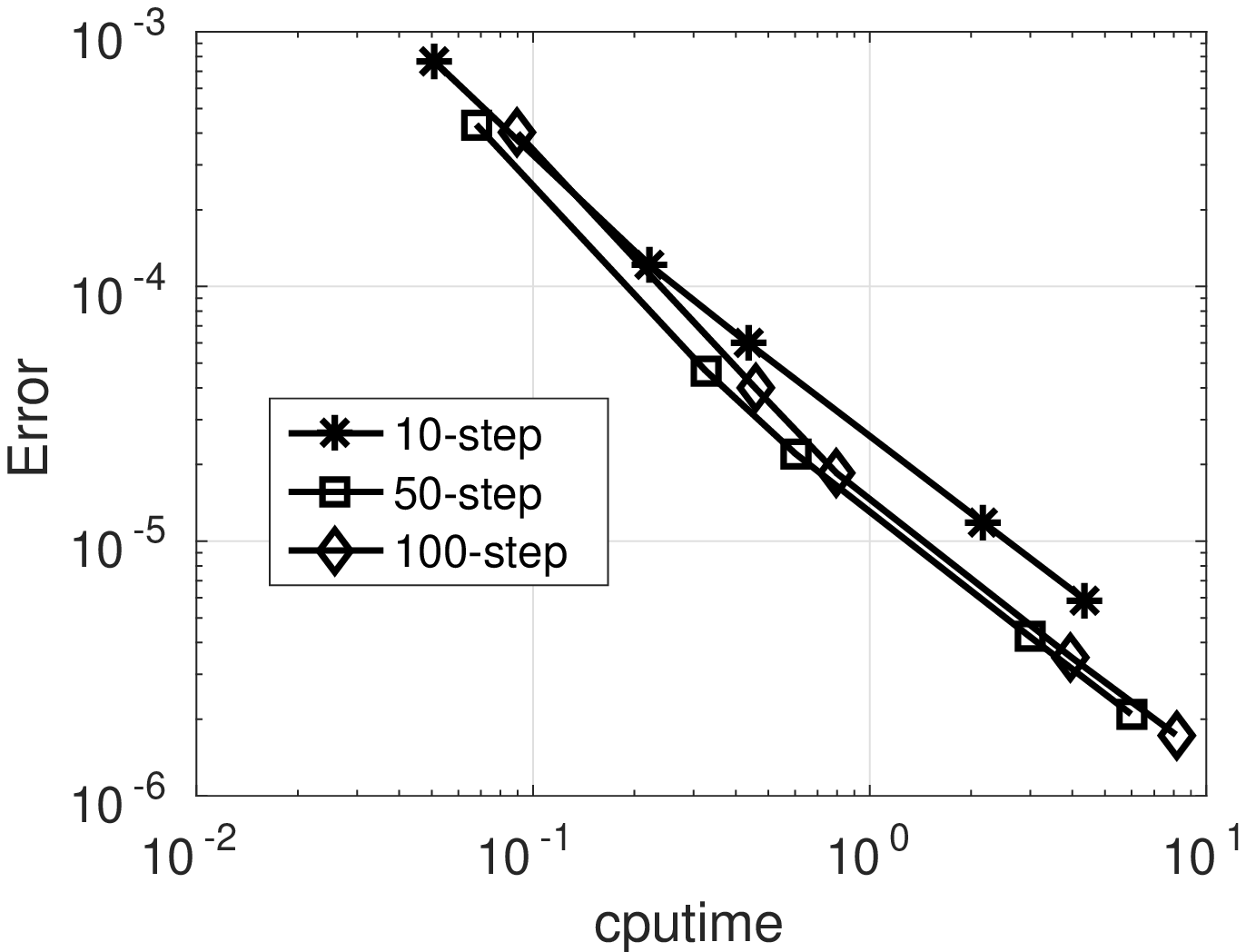}
\end{minipage}
\caption[Results for the substepping schemes, AC system]{Results for the substepping schemes applied to the Allen-Cahn type system. a) and c) display Estimated error against timestep $\Delta t$. b) and d) display estimated error against cputime, showing efficiency.}
\label{AC1-steps}
\end{figure}

\begin{figure}[h]
\centering
\begin{minipage}[b]{0.49\linewidth}
(a) \\
\includegraphics[width=0.99\columnwidth]{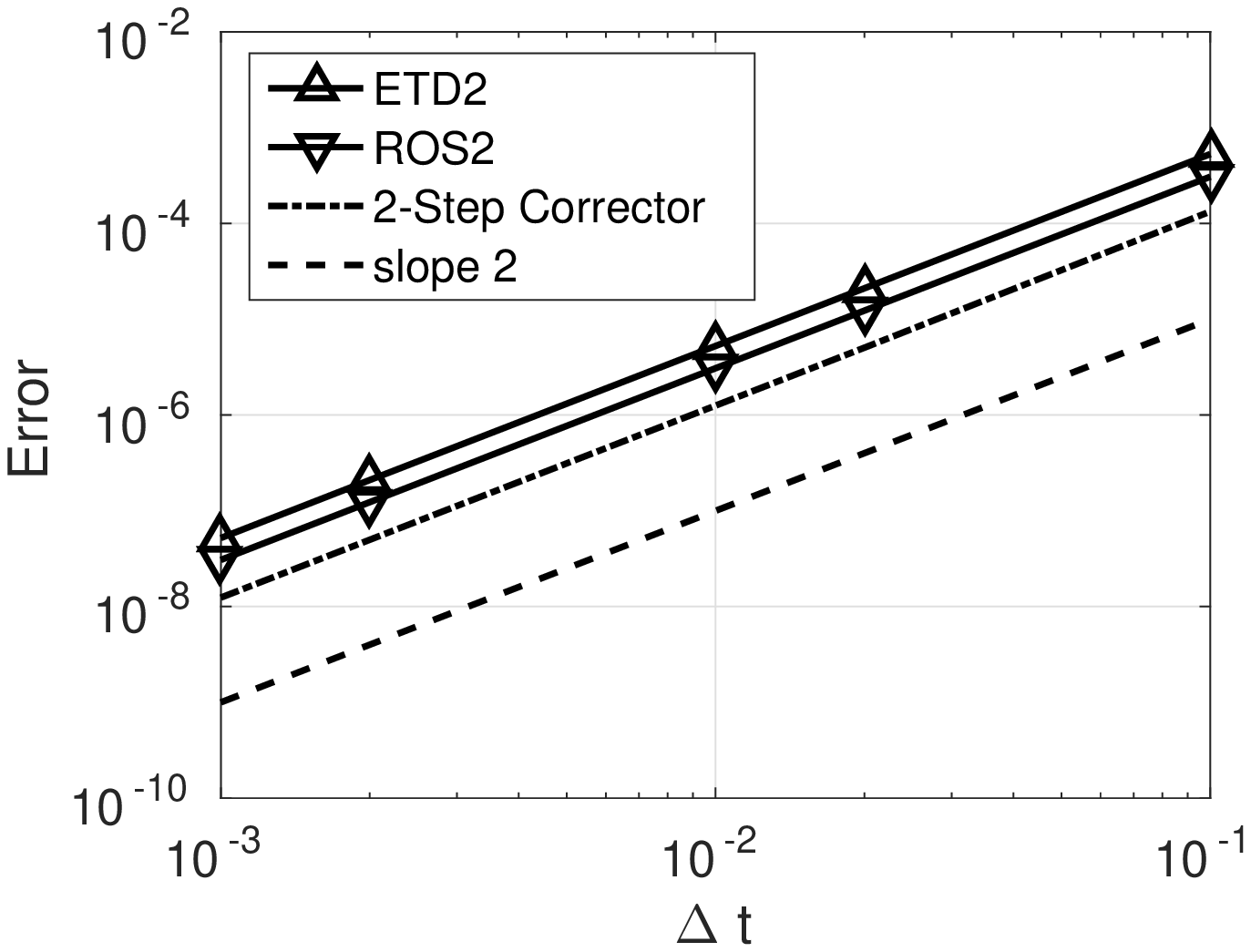}
\end{minipage}
\begin{minipage}[b]{0.49\linewidth}
(b) \\
\includegraphics[width=0.99\columnwidth]{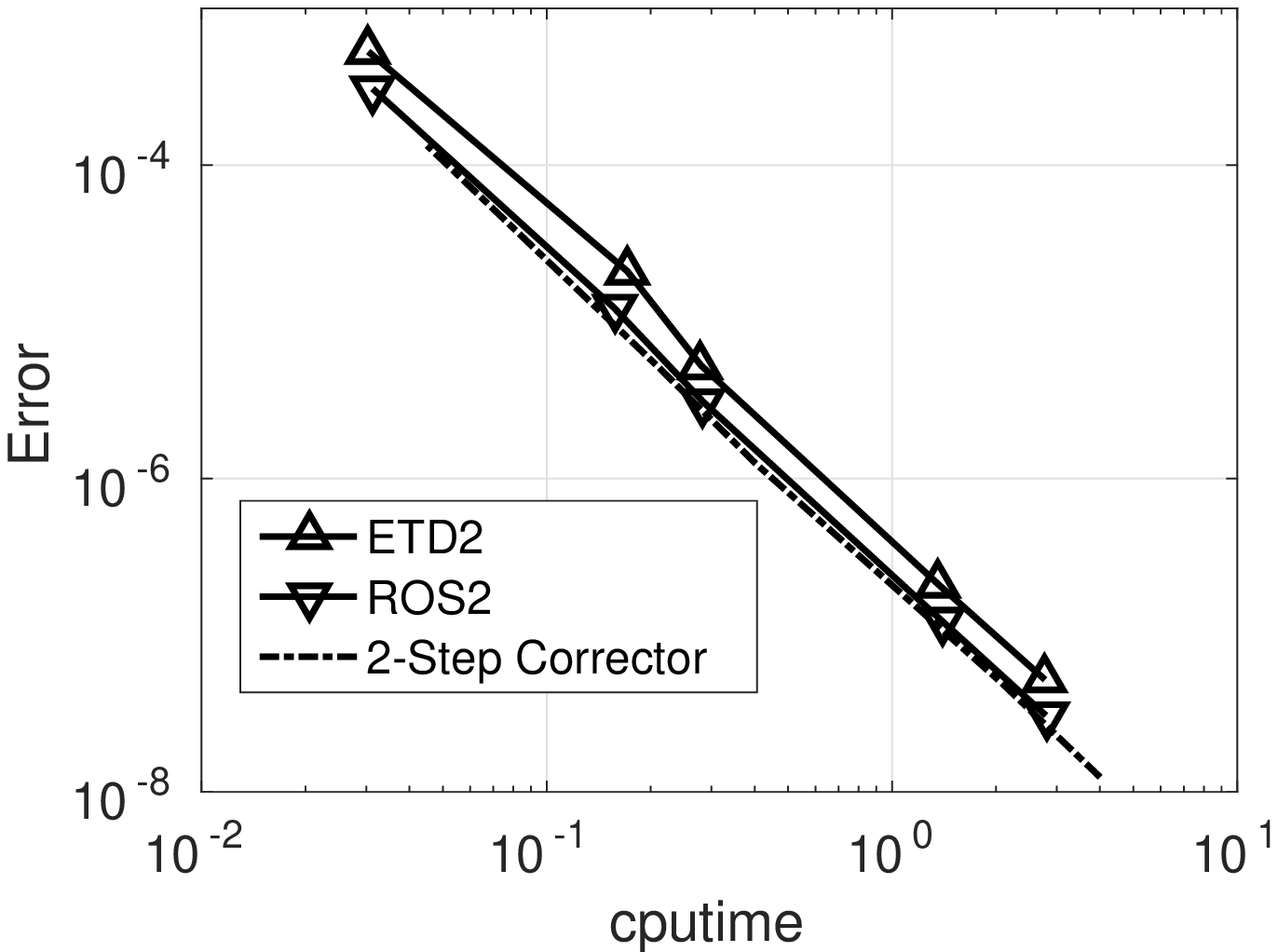}
\end{minipage}
\caption[Results for the second order substepping scheme, AC system]{AC system, Comparing the second order recycling-corrector scheme with ETD2 and ROS2. a) Estimated error against timestep $\Delta t$. b) Estimated error against cputime.}
\label{AC1-O2}
\end{figure}

\begin{figure}[h]
\begin{minipage}[b]{0.49\linewidth}
(a)\\
\includegraphics[width=0.99\columnwidth]{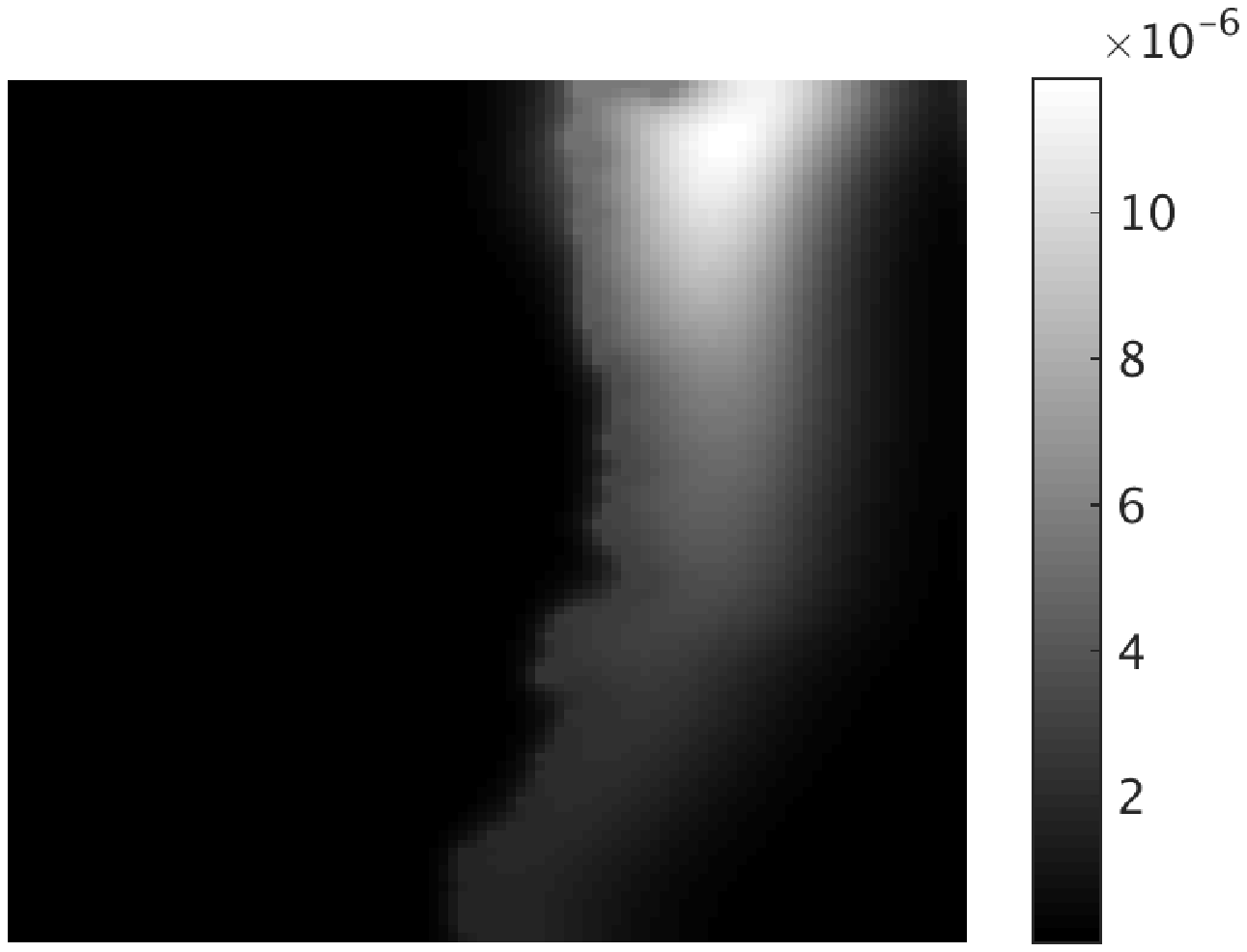}
\end{minipage}
\begin{minipage}[b]{0.49\linewidth}
(b)\\
\includegraphics[width=0.99\columnwidth]{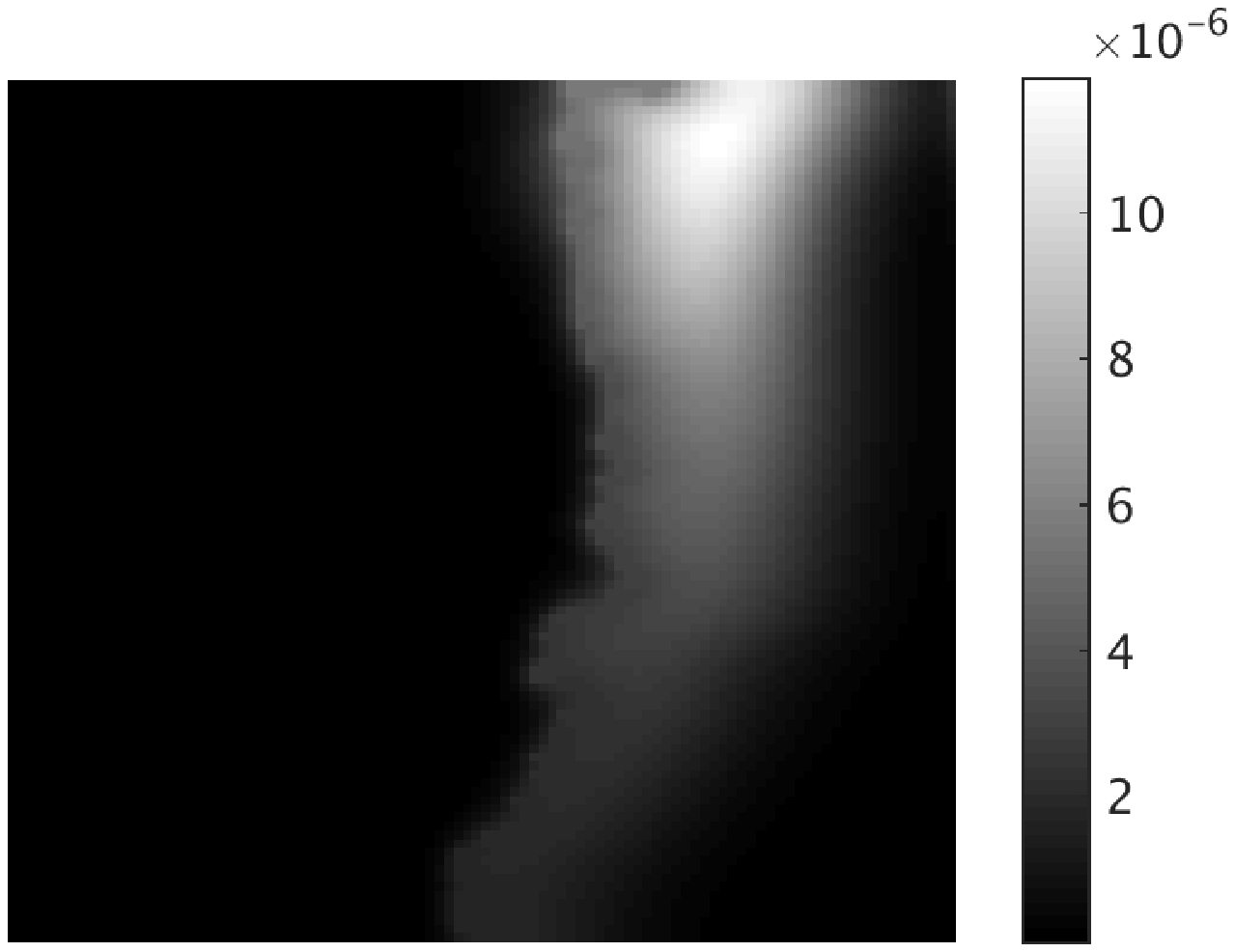}
\end{minipage}
\caption[Fracture system with Langmuir reaction]{The final state of the fracture system with Langmuir type reaction. a) Result produced by the 2-step scheme with $\Delta t = 2.4 \times 10^{-4}$. b) Result produced by ETD2 with $\Delta t = 2.4 \times 10^{-5}$.}
\label{frac-lang-show}
\end{figure}

\begin{figure}[h]
\centering
\begin{minipage}[b]{0.49\linewidth}
(a)  \\
  \includegraphics[width=0.99\columnwidth]{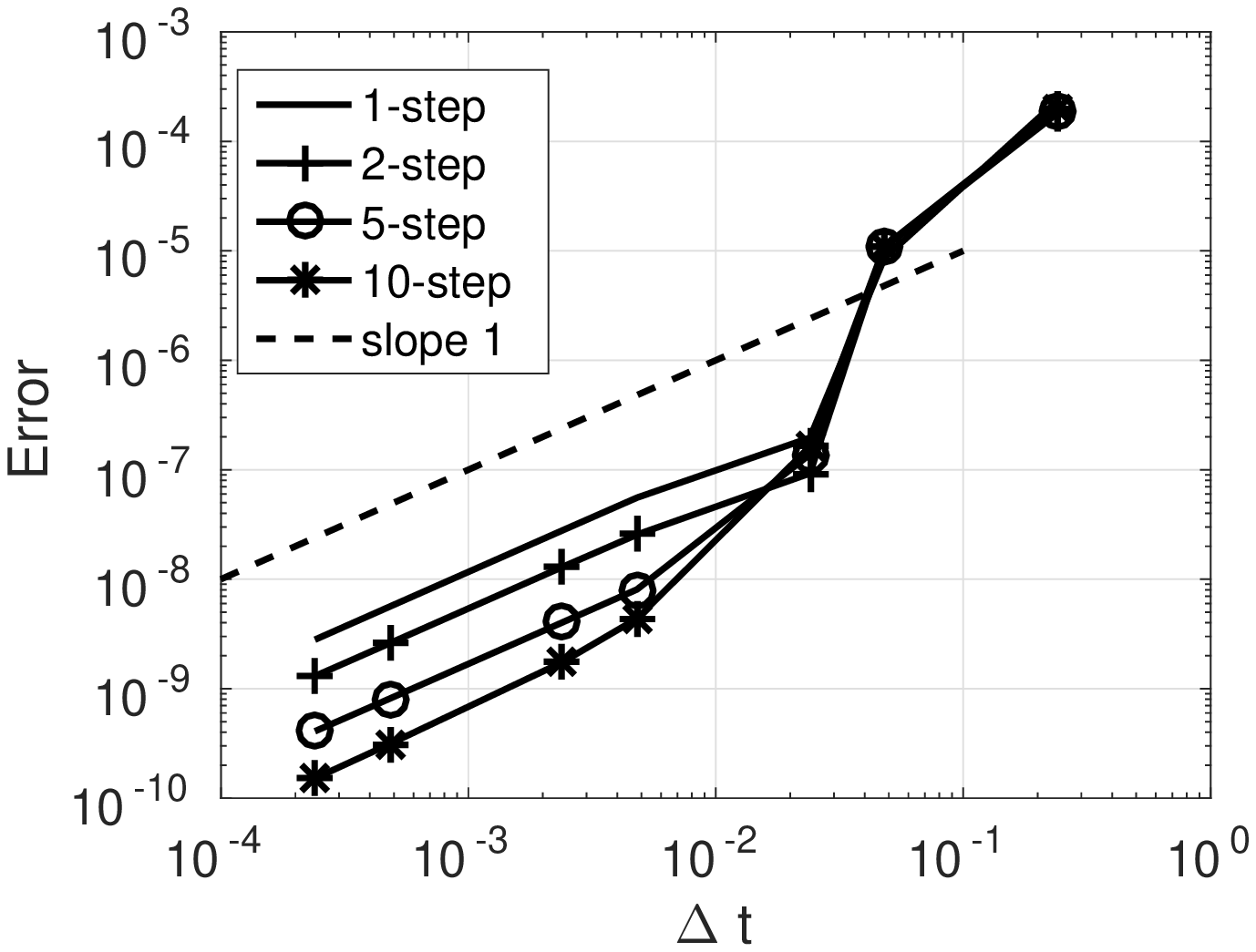}
\end{minipage}
\begin{minipage}[b]{0.49\linewidth}
(b) \\
\includegraphics[width=0.99\columnwidth]{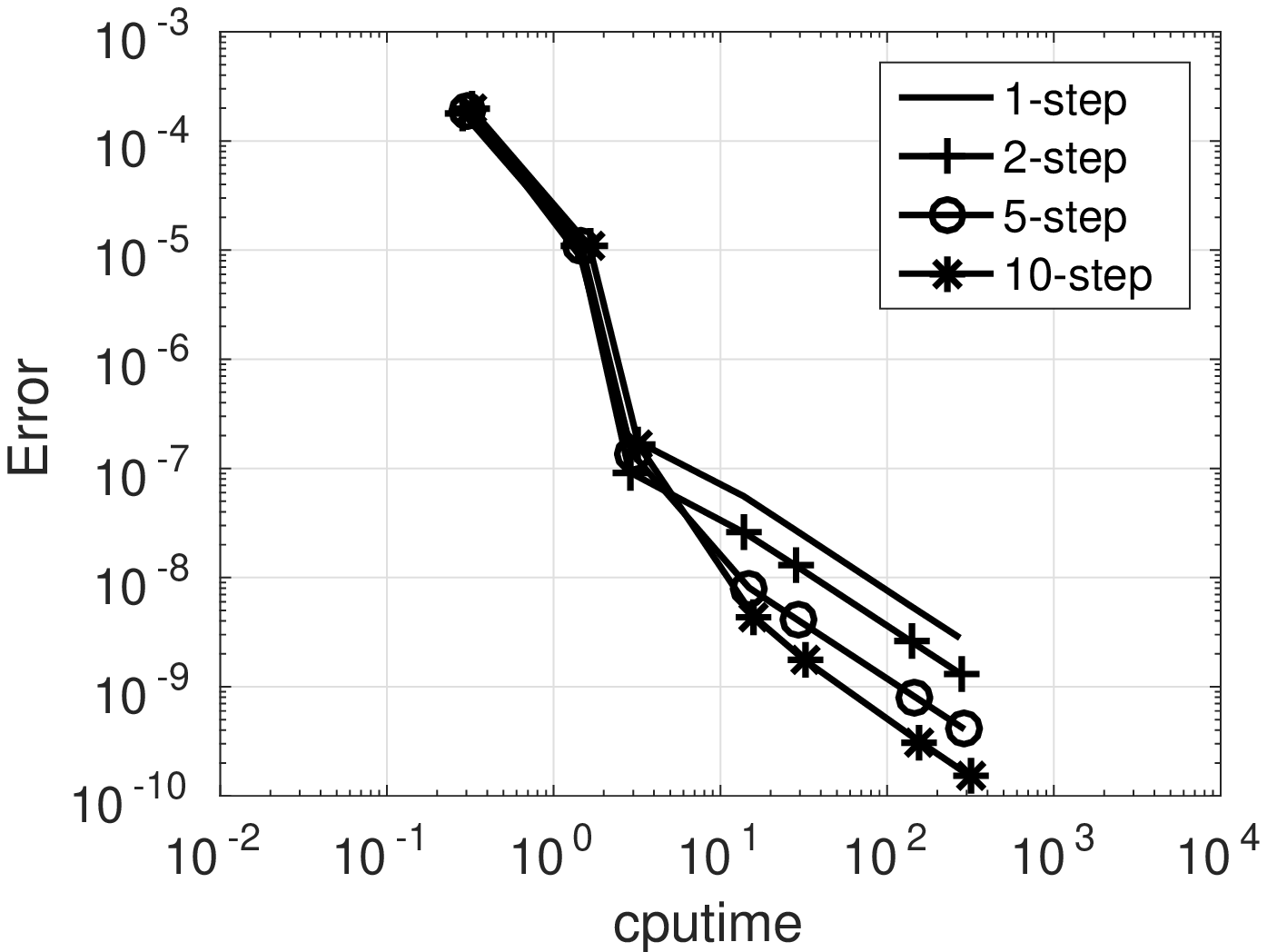}
\end{minipage}
\begin{minipage}[b]{0.49\linewidth}
(c) \\
\includegraphics[width=0.99\columnwidth]{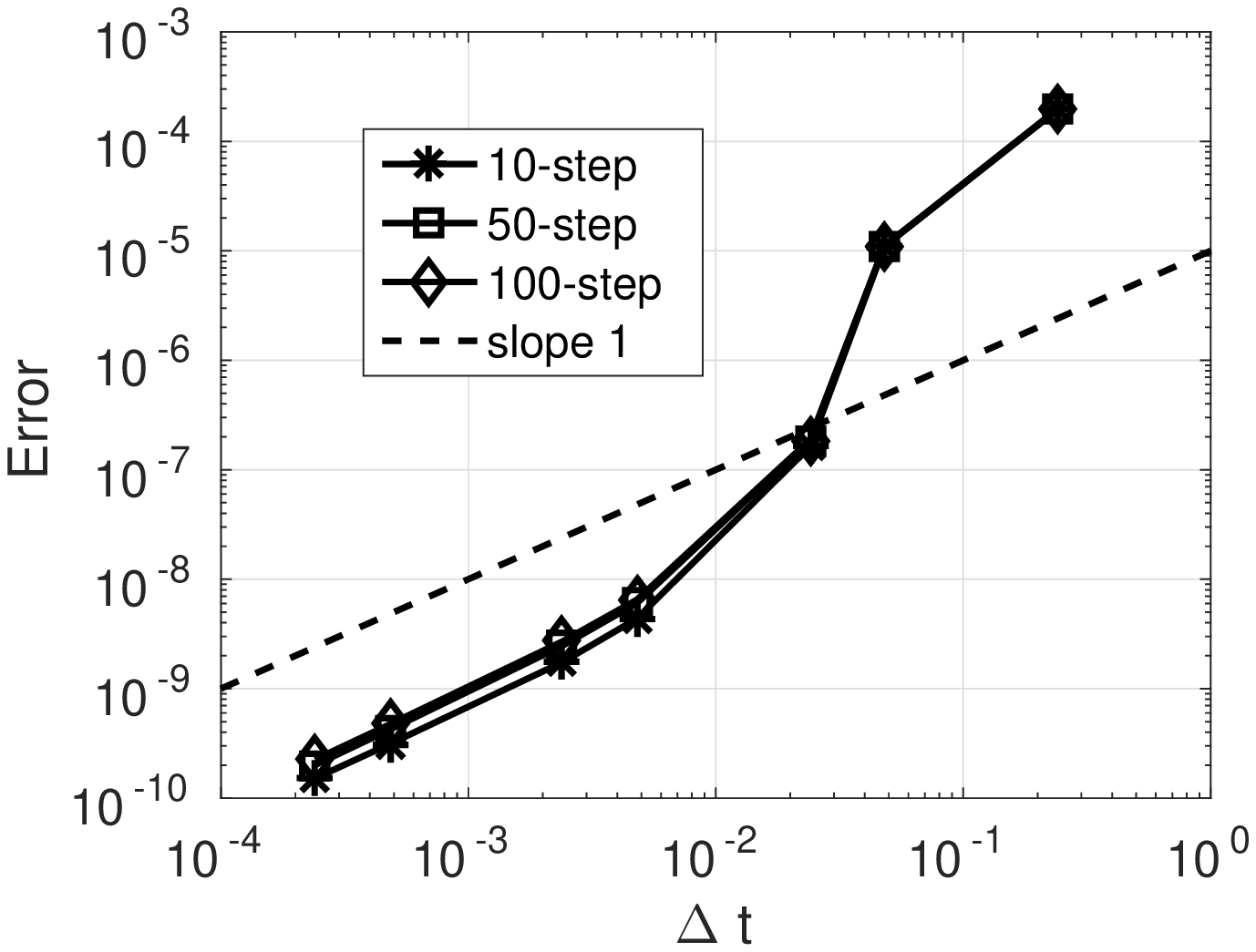}
\end{minipage}
\begin{minipage}[b]{0.49\linewidth}
(d) \\
\includegraphics[width=0.99\columnwidth]{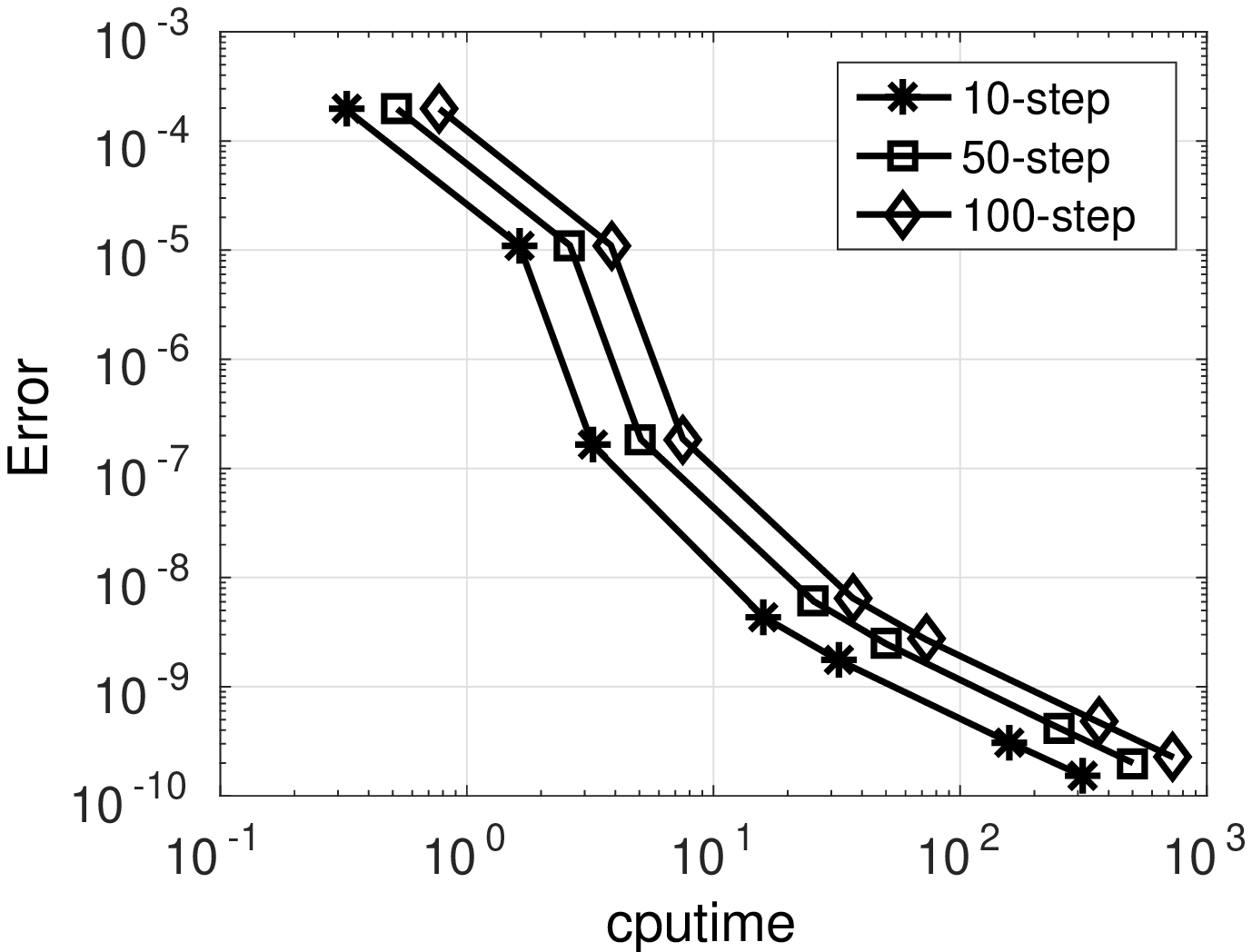}
\end{minipage}

\caption[Results for the substepping schemes, Langmuir type reaction system]{Results for the substepping schemes applied to the Langmuir type reaction system. a) and c) display Estimated error against timestep $\Delta t$. b) and d) display estimated error against cputime, showing efficiency. In c), points for the $50$ step scheme are marked with circles, and points with the $100$ step scheme are marked with triangles, to help distinguish the (very similar) results for the two schemes. This is also done in plot d) for consistency.}
\label{frac-lang-steps}
\end{figure}

\begin{figure}[h]
\centering
\begin{minipage}[b]{0.49\linewidth}
(a) \\
\includegraphics[width=0.99\columnwidth]{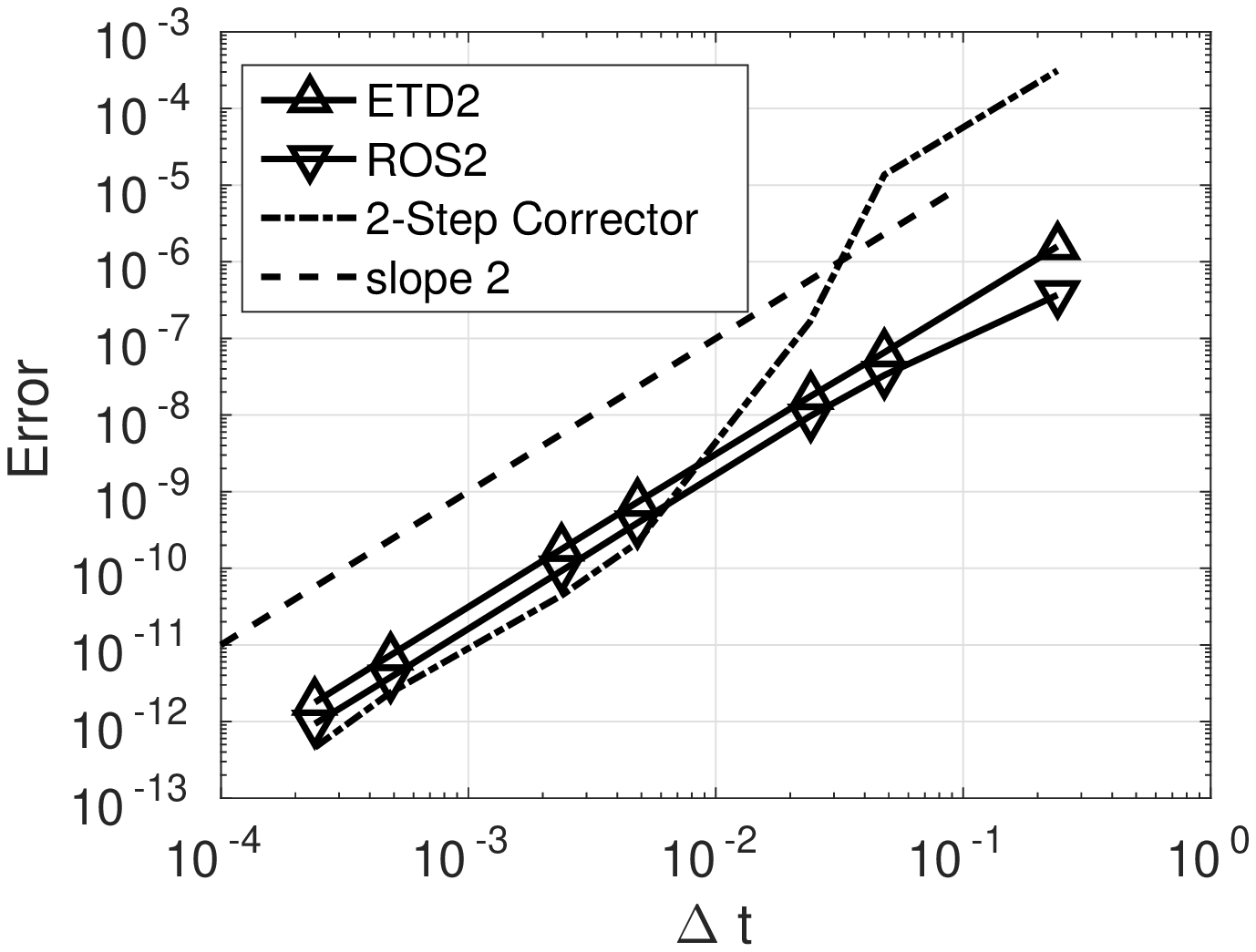}
\end{minipage}
\begin{minipage}[b]{0.49\linewidth}
(b) \\
\includegraphics[width=0.99\columnwidth]{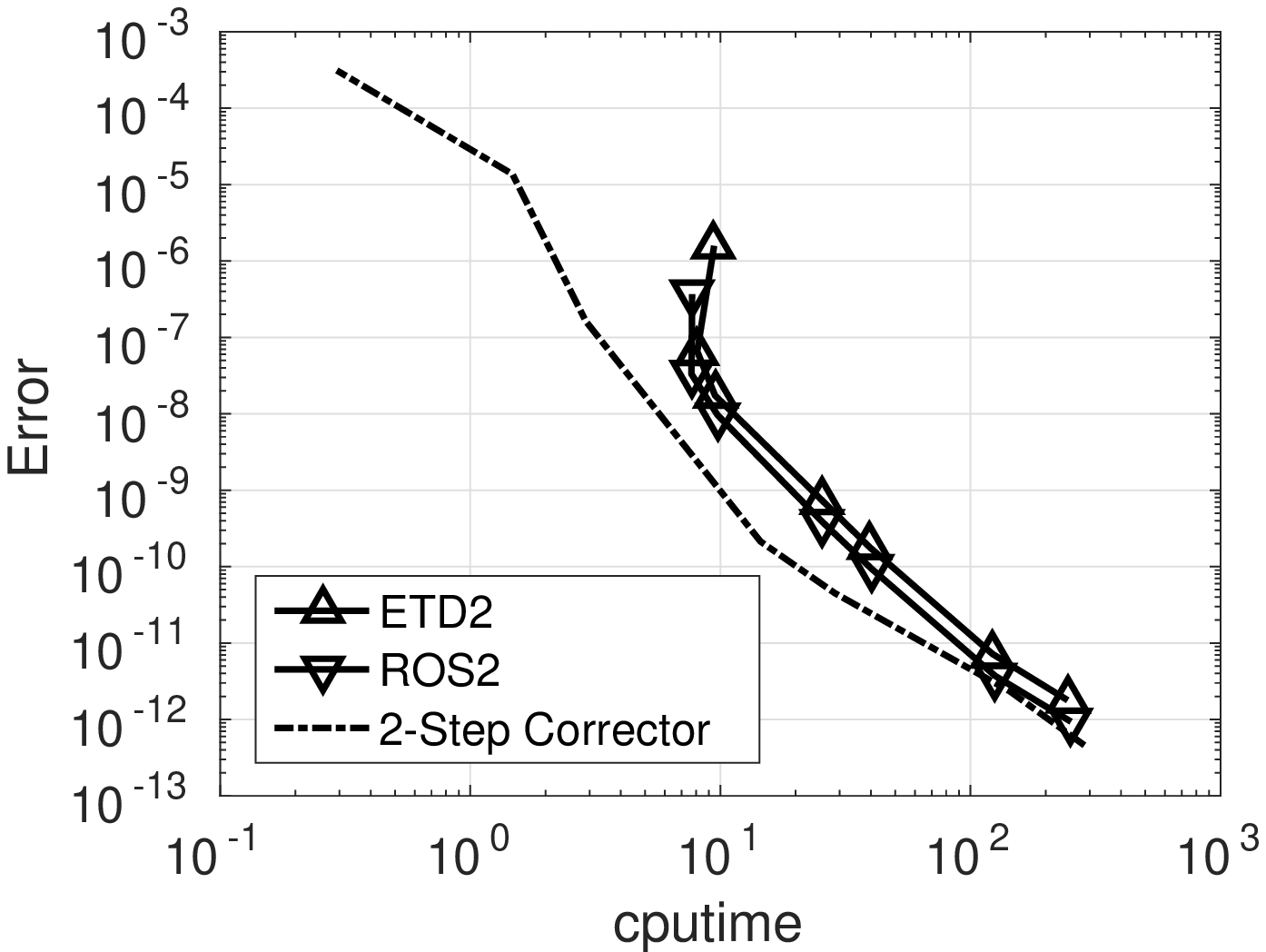}
\end{minipage}
\caption[Results for the second order substepping scheme, Langmuir type reaction system]{Langmuir type reaction system, Comparing the second order recycling-corrector scheme with ETD2 and ROS2. a) Estimated error against timestep $\Delta t$. b) Estimated error against cputime.}
\label{frac-lang-cor}
\end{figure}

\begin{figure}[h]
\centering
\begin{minipage}[b]{0.49\linewidth}
(a) \\
\includegraphics[width=0.99\columnwidth]{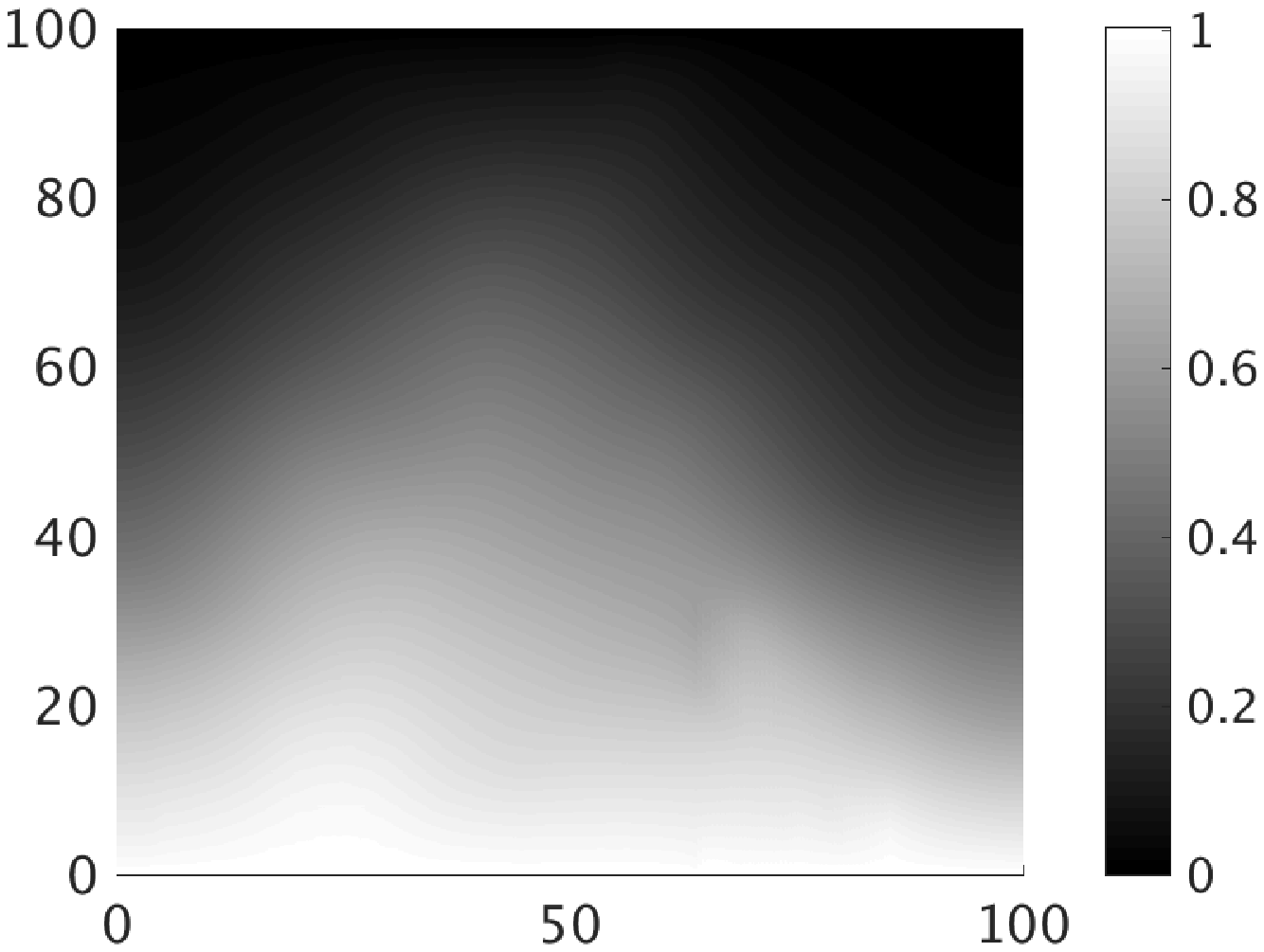}
\end{minipage}
\begin{minipage}[b]{0.49\linewidth}
(b) \\
\includegraphics[width=0.99\columnwidth]{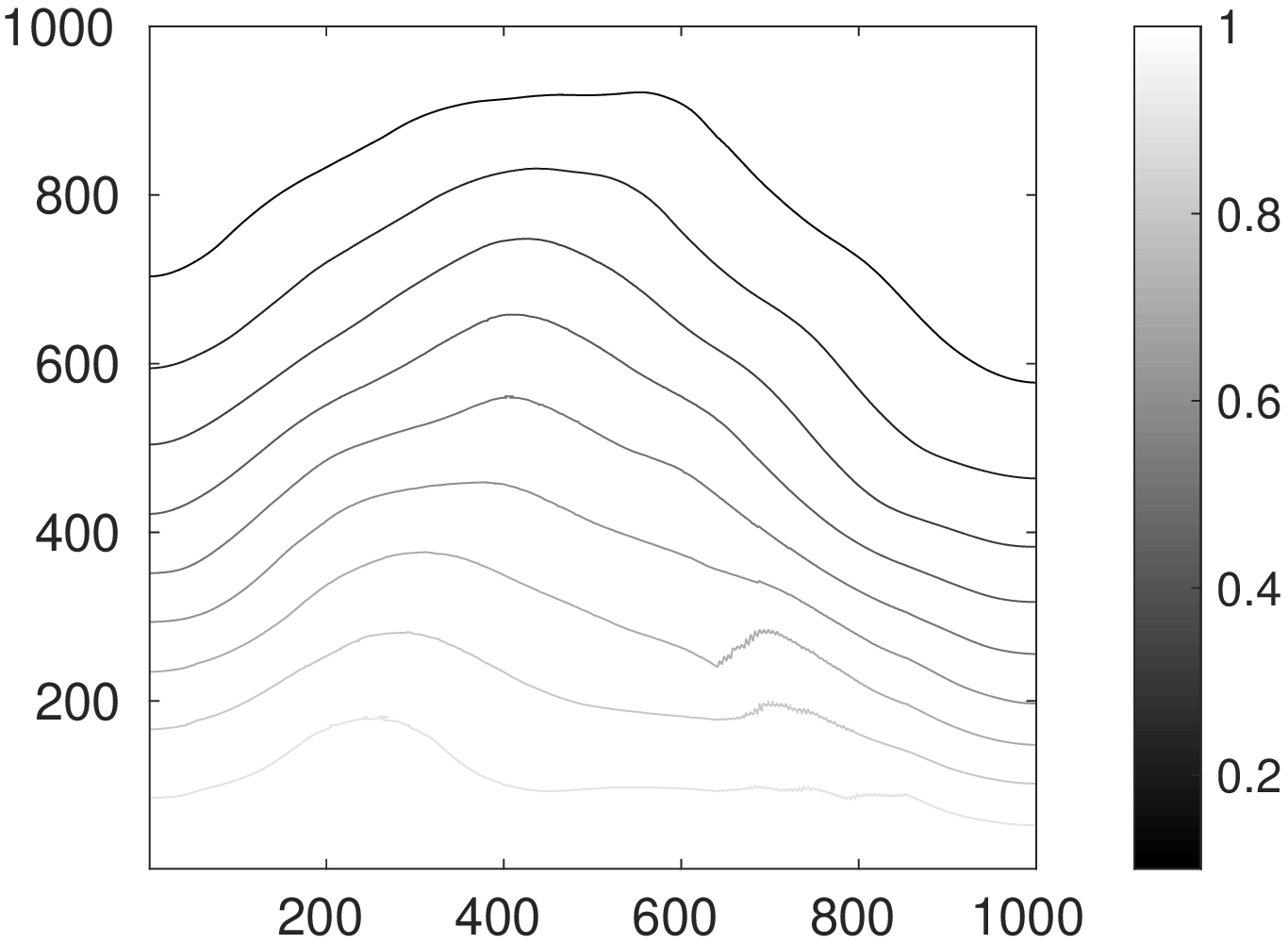}
\end{minipage}
\caption{Result for the example in \secref{large}, in which solute
  enters through the lower boundary and flows according to a random
  velocity field. Produced by the 10-step recycling scheme with
  $\Delta t = 0.2441$, i.e. $2048$ steps. a) Shows the system at the final time $T=500$; the axes indicates the physical dimensions (i.e., the domain is $100 \times 100$ metres). b) Is a corresponding contour plot for clarity; the axes indicate the cells in the finite volume grid (i.e., the grid has $1000$ cells along each side).}
\label{giant-show}
\end{figure}

\begin{figure}[h]
\centering
\begin{minipage}[b]{0.49\linewidth}
(a) \\
\includegraphics[width=0.99\columnwidth]{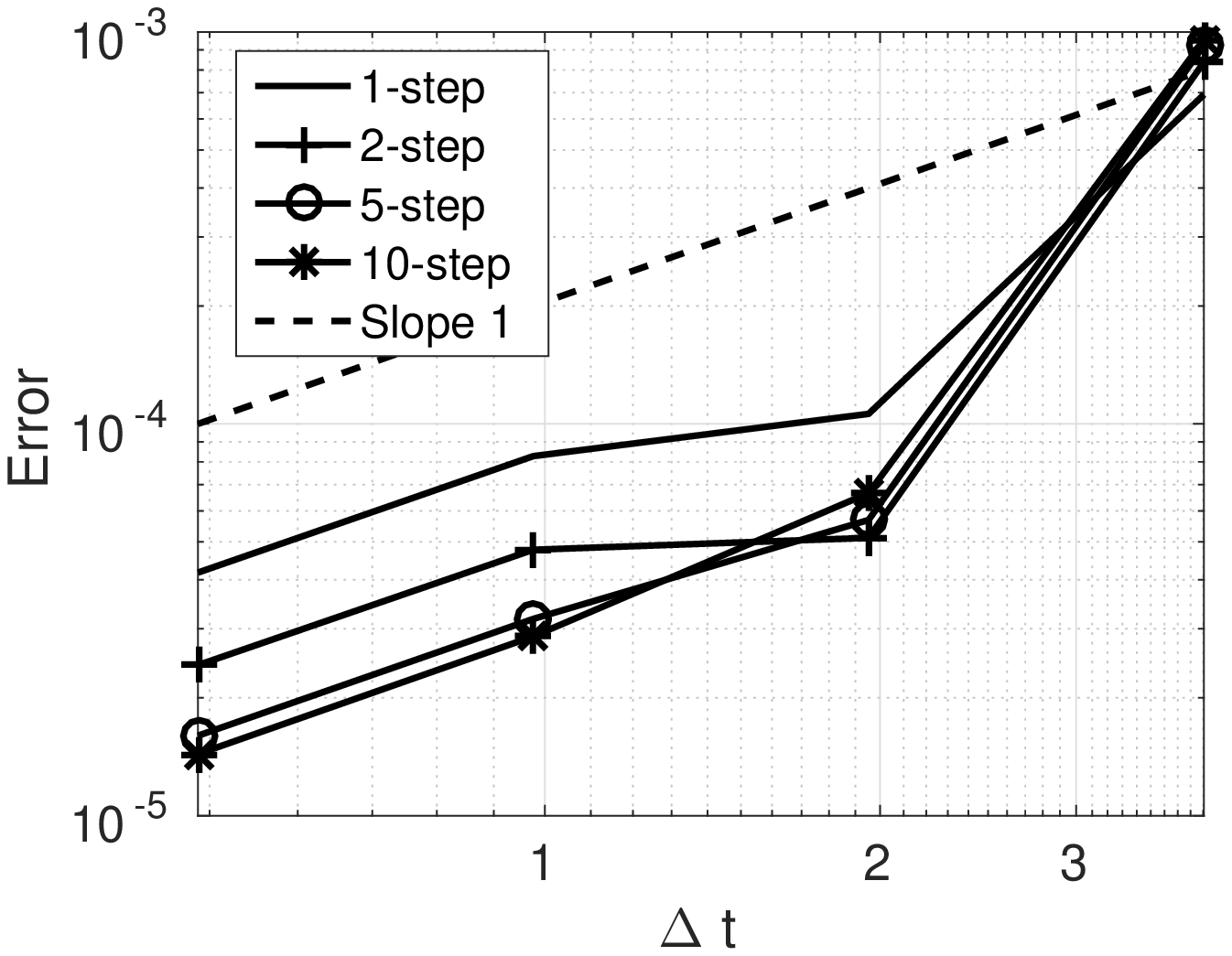}
\end{minipage}
\begin{minipage}[b]{0.49\linewidth}
(b) \\
\includegraphics[width=0.99\columnwidth]{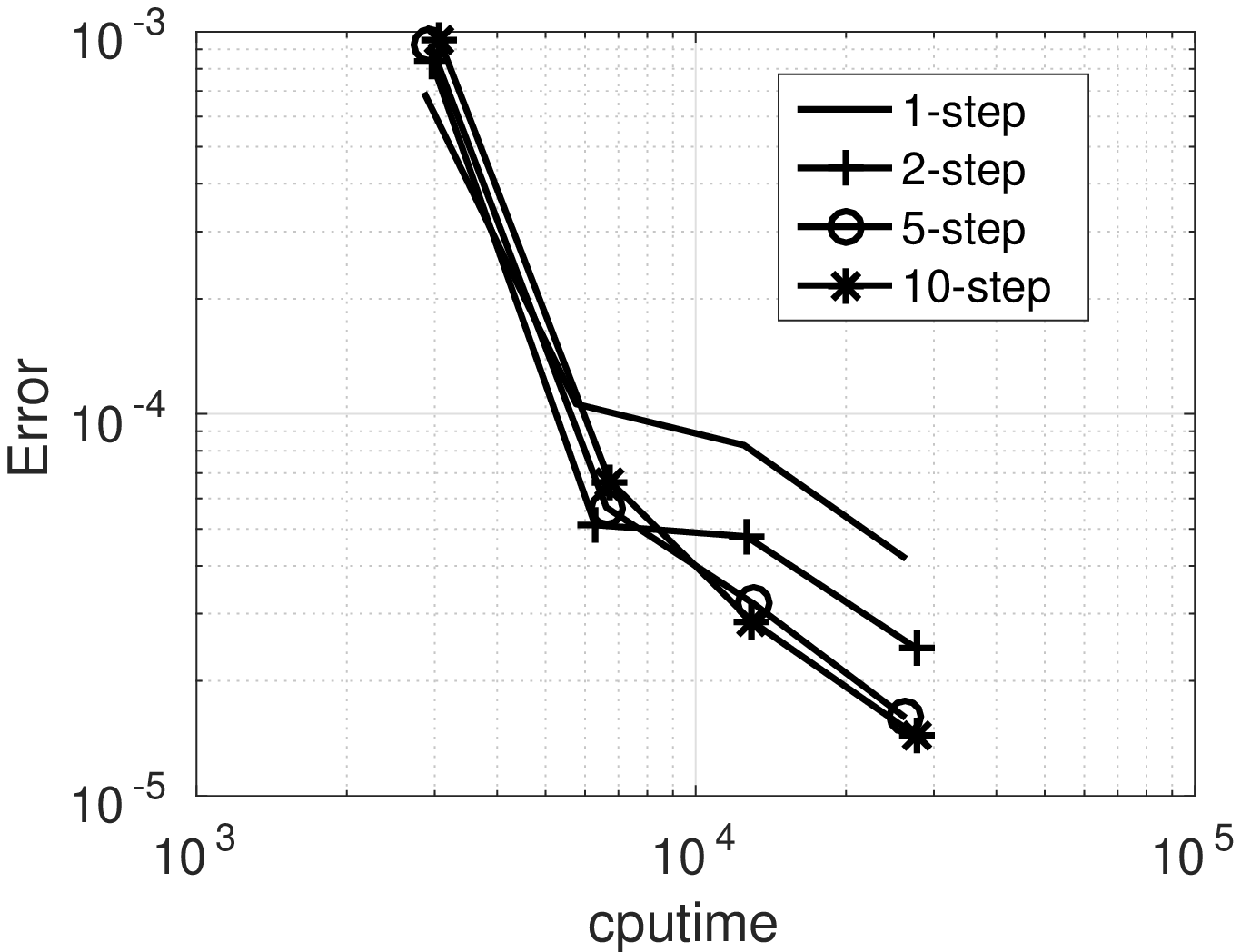}
\end{minipage}
\caption{Results for the  substepping scheme applied to the large
  Langmuir type reaction system in \secref{large}. (a) Estimated
  errors against timestep $\Delta t$ and (b) displays estimated error
  against cputime, showing efficiency. Note that for the largest
  timestep the error is dominated by the Krylov error as $m$ is too small for the given $\Delta t$ (c.f. Assumption \ref{asskryerr} )}  
\label{giant-results}
\end{figure}

\end{document}